\documentclass[12pt]{article}
\usepackage{fullpage}
\usepackage[english]{babel}
\usepackage{amsmath,amsthm,amssymb}
\usepackage{amsfonts}
\usepackage{graphicx}
\usepackage{epstopdf}
\usepackage{amsbsy}
\usepackage{algorithm}
\usepackage{algpseudocode}
\usepackage{subfig}
\usepackage{dsfont}
\usepackage{color}
\usepackage{mathrsfs}
\usepackage{bm}
\usepackage{bbm}
\usepackage[flushleft]{threeparttable}
\usepackage{relsize}

\newcommand{\Z}{\mathbb{Z}}
\newcommand{\R}{\mathbb{R}}
\newcommand{\E}{\mathbb{E}}
\newcommand{\I}{\mathcal{I}}

\usepackage{hyperref}

\newtheorem{theorem}{Theorem}[section]
\newtheorem{corollary}{Corollary}
\newtheorem{lemma}{Lemma}
\newtheorem{claim}{Claim}
\newtheorem{proposition}{Proposition}
\newtheorem{definition}{Definition}

\newtheorem{assumption}{Assumption}
\numberwithin{corollary}{section}
\numberwithin{lemma}{section}
\numberwithin{proposition}{section}
\numberwithin{definition}{section}
\numberwithin{remark}{section}
\numberwithin{equation}{section}
\numberwithin{table}{section}
\numberwithin{example}{section}
\numberwithin{assumption}{section}

\DeclareMathOperator*{\argmax}{arg\,max}
\allowdisplaybreaks

\makeatletter
\let\save@mathaccent\mathaccent
\newcommand*\if@single[3]{%
  \setbox0\hbox{${\mathaccent"0362{#1}}^H$}%
  \setbox2\hbox{${\mathaccent"0362{\kern0pt#1}}^H$}%
  \ifdim\ht0=\ht2 #3\else #2\fi
  }
\newcommand*\rel@kern[1]{\kern#1\dimexpr\macc@kerna}
\newcommand*\widebar[1]{\@ifnextchar^{{\wide@bar{#1}{0}}}
{\wide@bar{#1}{1}}}
\newcommand*\wide@bar[2]{\if@single{#1}{\wide@bar@{#1}{#2}{1}}
{\wide@bar@{#1}{#2}{2}}}
\newcommand*\wide@bar@[3]{%
  \begingroup
  \def\mathaccent##1##2{%
    \let\mathaccent\save@mathaccent
    \if#32 \let\macc@nucleus\first@char \fi
    \setbox\z@\hbox{$\macc@style{\macc@nucleus}_{}$}%
    \setbox\tw@\hbox{$\macc@style{\macc@nucleus}{}_{}$}%
    \dimen@\wd\tw@
    \advance\dimen@-\wd\z@
    \divide\dimen@ 3
    \@tempdima\wd\tw@
    \advance\@tempdima-\scriptspace
    \divide\@tempdima 10
    \advance\dimen@-\@tempdima
    \ifdim\dimen@>\z@ \dimen@0pt\fi
    \rel@kern{0.6}\kern-\dimen@
    \if#31
      \overline{\rel@kern{-0.6}\kern\dimen@\macc@nucleus\rel@kern{0.4}
      \kern\dimen@}%
      \advance\dimen@0.4\dimexpr\macc@kerna
      \let\final@kern#2%
      \ifdim\dimen@<\z@ \let\final@kern1\fi
      \if\final@kern1 \kern-\dimen@\fi
    \else
      \overline{\rel@kern{-0.6}\kern\dimen@#1}%
    \fi
  }%
  \macc@depth\@ne
  \let\math@bgroup\@empty \let\math@egroup\macc@set@skewchar
  \mathsurround\z@ \frozen@everymath{\mathgroup\macc@group\relax}%
  \macc@set@skewchar\relax
  \let\mathaccentV\macc@nested@a
  \if#31
    \macc@nested@a\relax111{#1}%
  \else
    \def\gobble@till@marker##1\endmarker{}%
    \futurelet\first@char\gobble@till@marker#1\endmarker
    \ifcat\noexpand\first@char A\else
      \def\first@char{}%
    \fi
    \macc@nested@a\relax111{\first@char}%
  \fi
  \endgroup
}
\makeatother

\begin{document}

\title{Analyzing and Provably Improving Fixed Budget Ranking and Selection Algorithms}
\author{Di Wu \quad Enlu Zhou\\
H. Milton Stewart School of Industrial and Systems Engineering,\\
Georgia Institute of Technology}
\maketitle

\begin{abstract}
This paper studies the fixed budget formulation of the Ranking and Selection (R\&S) problem with independent normal samples, where the goal is to investigate different algorithms' convergence rate in terms of their resulting probability of false selection (PFS). First, we reveal that for the well-known Optimal Computing Budget Allocation (OCBA) algorithm and its two variants, a constant initial sample size (independent of the total budget) only amounts to a sub-exponential (or even polynomial) convergence rate. After that, a modification is proposed to achieve an exponential convergence rate, where the improvement is shown by a finite-sample bound on the PFS as well as numerical results. Finally, we focus on a more tractable two-design case and explicitly characterize the large deviations rate of PFS for some simplified algorithms. Our analysis not only develops insights into the algorithms' properties, but also highlights several useful techniques for analyzing the convergence rate of fixed budget R\&S algorithms.
\end{abstract}

\section{Introduction}
Stochastic simulation has become one of the most effective approaches to modeling large, complex and stochastic systems arising in various fields such as transportation, finance, supply chain management, power and energy, etc. It has also been used in many applications to identify which system design is optimal under certain performance criterion (e.g., the expected cost). This leads to what is called simulation optimization or Optimization via Simulation (OvS). In particular, using simulations to identify the best design among a finite number of candidates, generally referred to as Ranking and Selection (R\&S), is of great practical interest to study.

The research on R\&S is largely concerned with two related yet different formulations. The fixed confidence formulation aims to attain a target confidence level of the selected design's quality using as little simulation effort as possible, whereas the fixed budget formulation typically requires maximizing the probability of correct selection (PCS) under a fixed budget of simulation runs. For fixed confidence, a considerable amount of research effort goes to the indifference zone (IZ) formulation, which dates back at least to \cite{bechhofer1954single}. An IZ procedure guarantees selecting the best design with (frequentist) probability higher than certain level (e.g., 95\%), provided that the difference between the top-two designs is sufficiently large. Numerous efficient IZ procedures have been proposed in the simulation literature, including but are not limited to the KN procedure in \cite{kim2001fully}, the KVP and UVP procedures in \cite{jeff2006fully}, and the BIZ procedure in \cite{frazier2014fully}. We refer the reader to \cite{branke2005new} and \cite{kim2007recent} for excellent reviews of the development on this topic. In addition, the Bayesian approaches (see, e.g., \cite{chick2012sequential}) and the probably approximately correct (PAC) selection (see, e.g., \cite{ma2017efficient}) have also been studied in this stream of works. 

In this paper, we study the fixed budget formulation under a frequentist setting. In simulation optimization, the Optimal Computing Budget Allocation (OCBA) algorithm in \cite{chen2000simulation} is one of the most widely applied and studied algorithms. Although OCBA is usually derived under a normality assumption and asymptotic approximations, it is well known for its robust empirical performance even when the sample distributions are non-normal. Moreover, its key allocation rule can be formally justified from the perspective of the large deviations theory (see, e.g., \cite{glynn2004large}). However, a major criticism is that a theoretical performance guarantee is still lacking to this date, mostly due to the difficulties in characterizing the PCS for sequential sampling algorithms. On the other hand, in the Multi-Armed Bandit literature, the same problem has been studied under the name of ``Best-Arm Identification'', where the Successive Rejects (SR) algorithm proposed in \cite{audibert2010best} currently stands as one of the best algorithms. Built on a framework of sequential elimination, SR not only achieves good performance but also allows a finite-sample bound to be derived. Furthermore, \cite{carpentier2016tight} showed that SR can match the optimal rate up to some constant in the Bernoulli setting. Nevertheless, SR's performance under general distributions has not yet been studied. Bayesian methods are also gaining momentum recently. For example, the simple Bayesian algorithms proposed in \cite{russo2016simple} were shown to achieve the optimal posterior convergence rate. However, there was no guarantee on the frequentist performance. The follow-up work, \cite{qin2017improving}, improved the Expected Improvement method and provided a frequentist bound, but the guarantee was for the fixed confidence setting.

Among the aforementioned algorithms, OCBA inarguably has the most variants and extensions. It has been extended to multi-objective optimization (\cite{lee2004optimal}), finding simplest good designs (\cite{jia2013efficient}), R\&S under input uncertainty (\cite{gao2017robust}), optimizing expected opportunity cost (\cite{gao2017new}) and many others. Meanwhile, there are attempts to approach the problem from different perspectives. For example, \cite{peng2018ranking} considered fixed budget R\&S under a Bayesian framework and formulated the problem as a Markov Decision Process, allowing a Bellman equation and an approximately optimal allocation to be derived. Also interestingly, \cite{ryzhov2016convergence} revealed that some variants of the Expected Improvement methods essentially have the same allocation as the OCBA methodology. Nonetheless, as was mentioned in \cite{lee2010review} as one of the open challenges, ``there are no theoretical proof to show how good the finite-time performance of OCBA is with respect to the real problem''.

The purpose of this paper is to better understand existing algorithms' behavior through rigorous analysis, and develop insights for improving their performance. In particular, our work highlights convergence analysis on the OCBA algorithm and some of its variants, where the convergence rate is measured in terms of the large deviations rate of the probability of false selection (PFS). A small portion of the results in this paper can be found in \cite{wu2018ocba}, which only improves two variants of OCBA in a simplified two-design setting. The current paper significantly extends \cite{wu2018ocba} by generalizing to the multiple-design case, and characterizing the LD rate for several other algorithms. Specifically, our contributions are summarized as follows.

\begin{enumerate}
\item
We show that for three OCBA-type algorithms including the original OCBA, a constant initial sample size only amounts to a sub-exponential (or even polynomial) convergence rate of PFS.
\item
By making the initial sample size increase linearly with the total budget, we improve the convergence rate to exponential, as is shown by a finite-sample bound on the PFS. The improvement is further validated via numerical experiments.
\item
As further exploration towards general convergence analysis, we exactly characterize the convergence rate of two simplified algorithms for a two-design case, where the results showcase some interesting insights as well as useful proof techniques.
\end{enumerate}

The rest of the paper is organized as follows. A brief review on the fixed budget R\&S problem is provided in Section \ref{sec2}. Section \ref{sec3} reveals the drawback of constant initial sample size for several OCBA-type algorithms, and proposes a modification to improve their convergence rate. Section \ref{sec4} conducts a preliminary study on convergence rate characterization by analyzing some simplified algorithms in a two-design case. Numerical results are presented in section \ref{sec5}, followed by conclusion and future work in section \ref{sec6}.

\section{Problem Formulation} \label{sec2}
Given a set of designs $\mathcal{I} = \{1,\ldots, K\}$, our goal is to select (without loss of generality) the one with the highest expected performance. Samples from simulating design $i$ are denoted by $X_{ir}$, where $r$ denotes the $r$th simulation run. Each design's expected performance is unknown, and is typically evaluated through multiple simulation runs and estimated by the sample mean
\begin{equation*}
\bar{X}_{i, N_i}  := \frac{1}{N_i} \sum_{r=1}^{N_i} X_{ir},
\end{equation*}
where $N_i$ is how many times design $i$ has been sampled/simulated. The subscript $N_i$ will be suppressed when there is no ambiguity. The true best and the observed best designs are denoted by
\begin{equation*}
b:= \argmax_{i \in \mathcal{I}} \mu_i, \qquad \hat{b}:= \argmax_{i \in \mathcal{I}} \bar{X}_i,
\end{equation*}
respectively. We make the following standard assumptions to avoid technicalities, where $\mathcal{N}$ stands for normal distribution and ``i.i.d.'' means ``independent and identically distributed''.

\begin{assumption} \label{assump1}
\quad
\begin{enumerate}
\item[(i)]
$K \ge 2$ and $\mu_i \neq \mu_j$ for any two different designs $i$ and $j$.
\item[(ii)]
For each design $i$, $\{X_{ir}\}$ are i.i.d. samples from $\mathcal{N}(\mu_i, \sigma_i^2)$, where $\sigma_i > 0$. The samples are also independent across different designs.
\end{enumerate}
\end{assumption}

Then, under a fixed budget $T$ of simulation runs, it is desired to maximize the probability of correct selection (PCS), which is defined as
\begin{equation*}
\text{PCS} := \mathbb{P}\left\{\hat{b} = b\right\} = \mathbb{P}\left\{\bigcap_{i \neq b} \left\{\bar{X}_b > \bar{X}_i \right\} \right\}.
\end{equation*}
We will also refer to $1 - \text{PCS}$ as the probability of false selection (PFS). The challenge of fixed budget R\&S problem lies in how to make the best use of a finite simulation budget to distinguish the best design from the rest. Numerous algorithms have been proposed to this end, and their performance is typically evaluated using two types of measures. 

The first type is asymptotic measures, which are often based on the large deviations (LD) theory. It has been shown in \cite{glynn2004large} that many algorithms have the following asymptotic property.
\begin{equation} \label{eq:LD}
- \lim_{T \rightarrow \infty} \frac{1}{T} \log \text{PFS}_{\mathcal{A}}(T, P) = R_{\mathcal{A}}(P),
\end{equation}
where $\mathcal{A}$ is an algorithm, $P$ is a problem instance, $\text{PFS}_{\mathcal{A}}(T, P)$ is the PFS of algorithm $\mathcal{A}$ applied to problem $P$ under budget $T$, and $R_{\mathcal{A}}(\cdot) \ge 0$ is called an LD rate function. For convenience, we say an algorithm $\mathcal{A}$ has an \emph{exponential convergence rate} if its PFS converges exponentially fast to 0, i.e., its LD rate $R_{\mathcal{A}}$ is positive. Asymptotically optimal algorithms have been derived by maximizing $R_{\mathcal{A}}$ (see, e.g., \cite{glynn2004large}), but it is an insufficient performance measure since it focuses primarily on the asymptotic performance. For example, all the terms in $\left\{e^{-T}, Te^{-T}, T^2e^{-T}, \ldots \right\}$ have the same LD rate according to (\ref{eq:LD}), yet they behave quite differently for small values of $T$. 

Measures of the second type emphasize more on the finite-sample performance. One approach is to approximate the PFS using tight bounds, but it could be remarkably difficult for algorithms that allocate the budget in a sequential style. Another approach is to plot out the PCS curve and visualize how fast it converges to 1 as $T$ increases. The main downside, however, is that such empirical results are problem-specific and may fail to represent the general performance of an algorithm. 

Bearing the pros and cons of these three approaches in mind, we will analyze and improve existing algorithms from an LD perspective, and substantiate the improvement using finite-samples bounds combined with numerical results. 

\section{Analyzing and Improving OCBA-type Algorithms} \label{sec3}
We begin by addressing the convergence rate of the well-known OCBA algorithm, which has not been rigorously studied to the best of our knowledge. Two of its variants, called OCBA-D and OCBA-R, are also studied for a comparison. Surprisingly, we discover that if all three algorithms follow the common practice of a constant initial sample size, then the PFS converges only sub-exponentially fast as opposed to the often implicitly conjectured exponential rate. A quick modification is then proposed to guarantee an exponential convergence rate.

\subsection{OCBA, OCBA-D, and OCBA-R} \label{sec3.1}
This section gives a brief introduction to OCBA and two of its variants, which we call OCBA-D and OCBA-R. To better describe the algorithms, we introduce the following notations. Let $S^2_{i,n}$ denote the standard sample variance estimator of $n$ i.i.d. samples from design $i$, and let $\ell$ denote the iteration number of the algorithms, where $\ell = 0$ corresponds to the initialization phase. The budget allocated to design $i$ at the end of the $\ell$th iteration is written as $N_i(\ell)$, and other quantities are defined accordingly. For example, we let $\bar{X}_i(\ell) := \bar{X}_{i, N_i(\ell)}$ and $S^2_i(\ell) := S^2_{i, N_i(\ell)}$. The OCBA algorithm is presented in Algorithm \ref{OCBA}.

\begin{algorithm}[t]
\caption{OCBA (Chen et al. (2000))}\label{OCBA}
\begin{algorithmic}[1]
\State {\bf Input}: $N_0 \ge 2, \Delta \ge 1, T \ge KN_0$.
\State {\bf Initialization:} Sample each design $N_0$ times and compute $\bar{X}_i(0)$ and $S_i^2(0)$. $N_i(0) \gets N_0$. $T'(0) \gets N_0  K + \Delta$. $\ell \gets 0$.
\While{$\sum_{i \in \mathcal{I}} N_i(\ell) < T$ and $T'(\ell) \le T$}
\State $\hat{b} \gets \argmax_{i \in \mathcal{I}} \bar{X}_i(\ell)$.
\State Compute $\hat{\alpha}_1(\ell), \ldots, \hat{\alpha}_K(\ell)$ using (\ref{eq3.1}).
\For{$i = 1, \ldots, K$}
\State Run $\max\{0, \lfloor \hat{\alpha}_i(\ell) T'(\ell) \rfloor - N_i(\ell)\}$ replications for design $i$.
\State $N_i(\ell+1) \gets \max\{N_i(\ell), \lfloor \hat{\alpha}_i(\ell) T' (\ell)\rfloor\}$. Compute $\bar{X}_i(\ell+1)$ and $S_i(\ell+1)$.
\EndFor
\State $\ell \gets \ell+1$.
\State $T'(\ell+1) \gets T'(\ell) + \Delta$.
\EndWhile
\State {\bf Output:} $\hat{b} = \argmax_{i \in \mathcal{I}} \bar{X}_i(\ell)$.
\end{algorithmic}
\end{algorithm}

OCBA has three input parameters: (i) $N_0 \ge 2$ is the size of samples for an initial estimation of each design's mean and variance; (ii) $\Delta \ge 1$ is the increment of available budget at each iteration; (iii) $T \ge KN_0$ is the total budget. An auxiliary variable, $T'$, is introduced to implement sequential allocation. The procedure begins with estimating each design's mean and variance using $N_0$ samples, where $T'$ is set to be $KN_0$. Then, at each iteration, the algorithm increases $T'$ by $\Delta$, and (re)computes the fractions $\hat{\alpha}_1(\ell), \ldots, \hat{\alpha}_K(\ell)$ according to the following equations.
\begin{equation} \label{eq3.1}
\hat{\beta}_i(\ell) = \begin{cases}
S_i^2(\ell) / [\bar{X}_{\hat{b}}(\ell) - \bar{X}_i(\ell)]^2 & \mbox{$\text{if }i \neq \hat{b}$}\\
S_{\hat{b}}(\ell)\sqrt{\sum_{i \neq \hat{b}} \hat{\beta}^2_i(\ell)/S_i^2(\ell)} & \mbox{o./w}
\end{cases},
\qquad
\hat{\alpha}_i(\ell) := \frac{\hat{\beta}_i(\ell)}{\sum_{i \in \mathcal{I}} \hat{\beta}_i(\ell)}.
\end{equation}
With the fractions computed, the algorithm tries to match its current $N_i$  with the target allocation $\lfloor \hat{\alpha}_i(\ell) T'\rfloor$ to the greatest possible extent: if $N_i$ is below the target, run additional simulations to match its target; otherwise, maintain the current $N_i$ since consumed budget cannot be refunded. All the mean and variance estimates are updated at the end of each iteration. The process continues iteratively until the total budget is depleted. Finally, the design with the highest sample mean is selected as the output.

Observe that two features of OCBA stand out from Algorithm \ref{OCBA}. The first one to notice is the allocation fractions specified by (\ref{eq3.1}), which is a plug-in estimate of
\begin{equation} \label{eq3.2}
\beta_i := \begin{cases}
\sigma_i^2 / (\mu_b - \mu_i)^2 & \mbox{$\text{if } i \neq b$}\\
\sigma_b \sqrt{\sum_{i\neq b} \beta_i^2/\sigma_i^2} &\mbox{o./w.}
\end{cases},
\qquad
\alpha_i := \frac{\beta_i}{\sum_{i \in \mathcal{I}} \beta_i}.
\end{equation}
The fractions in (\ref{eq3.2}) can be derived by asymptotically maximizing a lower bound of the PCS under a normality assumption (see, e.g., \cite{chen2000simulation}). Moreover, \cite{glynn2004large} showed that for algorithms using a deterministic allocation of $N_i = \lfloor \alpha_i T \rfloor$, such fractions approximately maximize the LD rate of PFS in the case of i.i.d. normal samples. The other feature is sequential allocation, which consists of incrementally allocating the budget, repeatedly updating the estimated fractions $\hat{\alpha}_i$, and asymptotically matching the true allocation fractions $\alpha_i$ as $T \rightarrow \infty$. Empirical evidence shows that sequential allocation may be the key to its good finite-sample performance, even though a quantitative analysis is not available due to its highly complex dynamics. In this paper, we attempt to better understand OCBA by studying its asymptotic behavior, and our results will also shed some light on its finite-sample performance.

\begin{algorithm}[htb]
\caption{OCBA-D}\label{OCBA-D}
\begin{algorithmic}[1]
\State {\bf Input}: $N_0 \ge 2, T \ge KN_0$.
\State {\bf Initialization:} Sample each design $N_0$ times and compute $\bar{X}_i(0)$ and $S_i^2(0)$. $N_i(0) \gets N_0$. $\ell \gets 0$.
\While{$\sum_{i \in \mathcal{I}} N_i(\ell) < T$}
\State Compute $\hat{\alpha}_1(\ell), \ldots, \hat{\alpha}_K(\ell)$ using (\ref{eq3.1}).
\State Run one replication for design $i^* = \argmax_{i \in \I} \left\{ \hat{\alpha}_i(\ell) / N_i(\ell) \right\}$.
\State $N_{i^*}(\ell+1) \gets N_{i^*}(\ell) + 1$. Compute $\bar{X}_{i^*}(\ell+1)$ and $S_{i^*}(\ell+1)$.
\State $\ell \gets \ell+1$.
\EndWhile
\State {\bf Output:} $\hat{b} = \argmax_{i \in \mathcal{I}} \bar{X}_i(\ell).$
\end{algorithmic}
\end{algorithm}

\begin{algorithm}[h]
\caption{OCBA-R}\label{OCBA-R}
\begin{algorithmic}[1]
\State {\bf Input}: $N_0 \ge 2, T \ge KN_0$.
\State {\bf Initialization:} Sample each design $N_0$ times and compute $\bar{X}_i(0)$ and $S_i^2(0)$. $N_i(0) \gets N_0$. $\ell \gets 0$.
\While{$\sum_{i \in \mathcal{I}} N_i(\ell) < T$}
\State Compute $\hat{\alpha}_1(\ell), \ldots, \hat{\alpha}_K(\ell)$ using (\ref{eq3.1}).
\State Draw an independent sample $U(\ell)$ from $\text{Uniform}(0,1)$.
\State Run one replication for design $i^* = \min\{k \mid  U(\ell) \le \sum_{i=1}^k \hat{\alpha}_i(\ell), 1 \le k \le K\}$.
\State $N_{i^*}(\ell+1) \gets N_{i^*}(\ell) + 1$. Compute $\bar{X}_{i^*}(\ell+1)$ and $S_{i^*}(\ell+1)$.
\State $\ell \gets \ell+1$.
\EndWhile
\State {\bf Output:} $\hat{b} = \argmax_{i \in \mathcal{I}} \bar{X}_i(\ell).$
\end{algorithmic}
\end{algorithm}

In addition to OCBA, we also consider variations on OCBA and propose two variants, OCBA-D and OCBA-R, which are presented in Algorithms \ref{OCBA-D} and \ref{OCBA-R}, respectively. The ``D'' and ``S'' stand for ``Deterministic'' and ``Randomized''. Both variants inherit the fractions in (\ref{eq3.1}) and are designed to be fully sequential, i.e., at each iteration only a single additional run is allocated to some design $i^*$. However, their difference lies in the way $i^*$ is chosen. For OCBA-D, $i^*$ corresponds to the design with the largest ratio $\hat{\alpha}_i(\ell) / N_i(\ell)$, where the ratio is roughly a measure of need for simulations: intuitively, an undersampled design is reflected by a larger ratio relative to the others'. In OCBA-R, $i^*$ is chosen randomly by using the fractions as a sampling distribution. In other words, conditional on the $\hat{\alpha}$ vector, the choice of $i^*$ is independent of everything else. In sum, all three algorithms are governed by the ``asymptotically optimal'' fractions given by (\ref{eq3.2}), except that they use different sequential allocation strategies to approximate such fractions.

We consider OCBA-D and OCBA-R for two reasons. First, fully sequential allocation and randomization are among the most natural forms of generalization to consider, examples including the most-starving version of OCBA (\cite{chen2011stochastic}) and the Top-two Sampling Algorithms (\cite{russo2016simple}). It is therefore important to know if any finding for OCBA also applies to these variants. Second, such variations can often make the algorithm behave more regularly and thus more amenable to analysis.

\subsection{Convergence Analysis} \label{sec3.2}
As a main contribution of this paper, we formally analyze the performance of OCBA, OCBA-D and OCBA-R. Firstly, we show that all three algorithms attain the ``asymptotically optimal'' allocation fractions given by (\ref{eq3.2}) as $T \rightarrow \infty$. Secondly, we reveal that despite the convergence of fractions, if the initial sample size $N_0$ is chosen as a constant independent of $T$, then these algorithms suffer from a sub-exponential convergence rate.

To put our work in perspective, \cite{glynn2004large} were among the first to study the asymptotics of fixed budget R\&S algorithms. They established that if an algorithm pre-specifies some fractions $\alpha_i>0$ and simply sets $N_i = \lfloor \alpha_i T \rfloor$, then the PFS converges exponentially fast under weak assumptions on the sample distributions' tails. In particular, if the samples are i.i.d. normal, then the fractions given by (\ref{eq3.2}) approximately maximize the LD rate of the PFS. Perhaps under the influence of such insights, there seems to be an implicit conjecture that algorithms which ``asymptotically'' attain the optimal allocation fractions, such as OCBA, should enjoy a similar LD rate to its static counterpart's, or at least guarantee exponential convergence. In what follows, we disprove this conjecture by using OCBA and the two proposed variants as counterexamples.

To set the basis for our major discovery, we link Algorithms \ref{OCBA}-\ref{OCBA-R} through the convergence of their actual allocation fractions $N_i(\ell) / \sum_{j} N_j(\ell)$. Observe that $T \rightarrow \infty$ if and only if $\ell \rightarrow \infty$, so we characterize such convergence in terms of $\ell$ for convenience. All the proofs omitted in the paper can be found in the electronic companion.
\begin{proposition} \label{prop3.1}
Let Assumption \ref{assump1} hold and denote ``almost surely'' by ``$a.s.$''. Then, for OCBA, OCBA-D and OCBA-R, the following holds.
\begin{enumerate}
\item[(i)]
$N_i(\ell) \rightarrow \infty$ a.s. as $\ell \rightarrow \infty$ for all $i \in \mathcal{I}$.
\item[(ii)]
$\hat{\alpha}_i(\ell) \rightarrow \alpha_i$ a.s. as $\ell \rightarrow \infty$ for all $i \in \mathcal{I}$.
\item[(iii)]
$N_i(\ell) / \sum_{j \in \mathcal{I}} N_j(\ell) \rightarrow \alpha_i$ a.s. as $\ell \rightarrow \infty$ for all $i \in \mathcal{I}$.
\end{enumerate}
\end{proposition}

Proposition \ref{prop3.1} is not surprising since all three algorithms are designed to approximate and match the true fractions $\alpha_i$ in (\ref{eq3.2}). It holds regardless of the value of $N_0$ (as long as $N_0 \ge 2$), because the algorithms are capable of correcting the estimation error from the initialization phase. For this reason, a small $N_0$ is often employed to leave room for better allocation flexibility in succeeding iterations. For example, a common suggestion for $N_0$ is between 5 and 20 (see, e.g., \cite{law1991simulation,bechhofer1995design}). Nevertheless, the following theorem suggests that a constant $N_0$ independent of $T$ can cause the PFS to converge rather slowly.

\begin{theorem} \label{thm3.1}
Let Assumption \ref{assump1} hold. If $N_0$ is chosen as a constant independent of $T$, then for OCBA and OCBA-D,
\begin{equation} \label{eq:lb1}
\text{PFS}(T) \ge C T^{-(K-1)(N_0 - 1)}, \quad \forall T \ge KN_0,
\end{equation}
for some constant $C>0$ independent of $T$. Also, for OCBA-R,
\begin{equation} \label{eq:lb2}
-\lim_{T \rightarrow \infty} \frac{1}{T}\log \text{PFS}(T) = 0.
\end{equation}
\end{theorem}

Theorem \ref{thm3.1} appears somewhat surprising, as it states that a constant initial sample size leads to at most a polynomial convergence rate for OCBA and OCBA-D, and a sub-exponential convergence rate for OCBA-R. At a high level, it implies that the initial estimation error, though vanishing as $T\rightarrow \infty$, does not decrease at a sufficiently fast rate. It also implies that the convergence of allocation fractions alone does not say much about how fast the PFS converges. Before showing Theorem \ref{thm3.1}, we present a few technical lemmas and describe the main idea behind the proof.
\begin{lemma} \label{lemma1}
Let $S_n$ be the sample standard deviation of $n$ i.i.d. normal samples with variance $\sigma^2$. Then, for any $0<x<\sigma$,
\begin{equation} \label{eq:lemma1.1}
\mathbb{P}(S_n \le \sigma -x) \le \exp\left\{-\frac{(n-1)}{4} \left[1 - \left(\frac{\sigma-x}{\sigma} \right)^2 \right]^2 \right\},
\end{equation}

\begin{equation} \label{eq:lemma1.2}
\mathbb{P}(S_n \ge \sigma+x) \le \exp\left(- \frac{(n-1)x^2}{4\sigma^2} \right), \quad \forall x >0.
\end{equation}
\end{lemma}

\begin{lemma} \label{lemma2}
Let $S^2_n$ be the sample variance of $n$ i.i.d. $\mathcal{N}(\mu, \sigma^2)$ random variables. Then, $\forall c>0, \exists \bar{\epsilon} \in (0, \sigma)$ such that
\begin{equation} \label{eq:lemma2.1}
\sum_{n \ge 2} \mathbb{P}\left\{S_n \le \sigma - \bar{\epsilon} \right\} \le c.
\end{equation}
\end{lemma}

\begin{lemma} \label{lemma3}
Let $S^2_n$ be the sample variance of $n$ i.i.d. $\mathcal{N}(\mu, \sigma^2)$ random variables. Then, $\forall a \in (0,b)$, where $b > 0$ is a constant, $\exists \mathcal{K}_b > 0$ such that
\begin{equation*}
\mathbb{P}\left\{S_n \le a \right\} \ge \left(\mathcal{K}_b a\right)^{n-1}.
\end{equation*}
\end{lemma}

Lemma \ref{lemma1} provides some basic tail bounds for the standard deviation estimator $S_n$, which can be used to prove Lemma \ref{lemma2}. Lemma \ref{lemma3} is the leading cause behind the polynomial convergence rate for OCBA and OCBA-D, as it points out that the left tail of $S_n$ converges to 0 only at a polynomial rate. We will present the proof of Theorem \ref{thm3.1} for OCBA, and the rest can be found in the electronic companion. To illustrate the main idea, consider an adversarial scenario for OCBA where
\begin{enumerate}
\item
After the initialization phase, $\hat{b}$ is some suboptimal design, e.g., design 2.
\item
The algorithm allocates all the remaining budget to design 2.
\item
The sample mean of design 2 beats all other designs' over all iterations. 
\end{enumerate}

In the scenario described above, we say that the algorithm ``freezes'' all the designs other than design 2, which only happens if the initial estimates for the ``frozen'' designs are highly inaccurate. For instance, we may consider a case where for all $i \neq 2$, $S_i(0)$ takes very small value and thus $\hat{\alpha}_i(\ell)$ is also tiny. This would trick the algorithm into greedily sampling design 2, while all the other designs' mean and variance estimates get no further update and thus stay inaccurate. To avoid technicalities, we further require design 2 to be the observed best design throughout the allocation process, so that $\hat{\alpha}_2(\ell)$ takes the same functional form for any iteration $\ell$ (recall from (\ref{eq3.1}) that $\hat{\alpha}_{\hat{b}}(\ell)$ has a different form than $\hat{\alpha}_i(\ell), i \neq \hat{b}$). The rest is to bound the probability of such an event from below, and show that it is not exponentially rare.

\proof{Proof of Theorem \ref{thm3.1} (OCBA).}
Assume without loss of generality that $\mu_1 > \mu_2 > \cdots > \mu_K$. For each design $i$, we will construct events $E_i(T)$ such that on $\mathcal{E}(T):=\bigcap_{i=1}^K E_i(T)$, a false selection always occurs. Without ambiguity, we will simply drop $T$ and write $\mathcal{E}$ and $E_i$ instead. To begin with, by Lemma \ref{lemma2} we can choose $\bar{\epsilon} \in (0, \sigma_2)$ such that $\sum_{n \ge N_0} \mathbb{P}\{S_{2,n} \le \sigma_2 - \bar{\epsilon}\} \le 1/4$. By a similar argument, there exists $\bar{\eta}>0$ such that $\sum_{n \ge N_0} \mathbb{P}\{ \bar{X}_{2,n} \le \mu_2 - \bar{\eta}\} \le 1/4$. Let
\begin{equation*}
E_2:= \left\{\bar{X}_{2,n} \ge \mu_2 - \bar{\eta}, \forall n \ge N_0 \right\} \cap \left\{S_{2,n} \ge \sigma_2 - \bar{\epsilon}, \forall n \ge N_0 \right\}.
\end{equation*}
Then, $\mathbb{P}(E_2) \ge 1- \frac{1}{4} -\frac{1}{4} = \frac{1}{2}$ by a union bound. For $i \neq 2$, we let
\begin{equation*}
E_i := \left\{\bar{X}_{i,N_0} \le \mu_2 - \bar{\eta} -1 \right\} \cap \left\{S_{i, N_0} \le N_0(\sigma_2 - \bar{\epsilon})/T \right\}.
\end{equation*}
We now show that $\bigcap_{i=1}^K E_i \subseteq \text{FS}$ by induction, where ``FS'' stands for the false selection event. Fix a sample path on $E_i$. Note that $\hat{b} = 2$ after the initialization phase. Assume that $\hat{b}=2$ at the end of the $(\ell-1)$th iteration, then at the $\ell$th iteration, for any $i \neq 2$,
\begin{align*}
\hat{\alpha}_{i}(\ell) = &\frac{S_i^2(\ell) / \hat{\delta}_{2,i}^2(\ell)}{\sum_{j \neq 2} \frac{S_j^2(\ell)}{\hat{\delta}_{2,j}^2(\ell)} + S_2(\ell) \sqrt{\sum_{j \neq 2} \frac{S_j^2(\ell)}{\hat{\delta}_{2,j}^4(\ell)}}}
\le \frac{S_i^2(\ell) / \hat{\delta}_{2,i}^2(\ell)}{S_2(\ell) \sqrt{\frac{S_i^2(\ell)}{\hat{\delta}_{2,i}^4(\ell)}}}\\
=&\frac{S_i(\ell)}{S_2(\ell)} \le \frac{N_0 (\sigma_2 - \bar{\epsilon})}{S_2(\ell) T},
\end{align*}
where $\hat{\delta}_{i,j} (\ell) := \bar{X}_i(\ell) - \bar{X}_j(\ell)$. From $E_2$ we have $S_2(\ell) \ge \sigma_2 - \bar{\epsilon}$, thus $\hat{\alpha}_i(\ell) T'(\ell) \le \hat{\alpha}_i(\ell)  T\le N_0$ for all $i \neq 2$ and only design 2 will get additional sample at step 7 of Algorithm \ref{OCBA}. Since $\bar{X}_2(\ell) \ge \mu - \bar{\eta} > \mu_2 - \bar{\eta} - 1 \ge \bar{X}_{i,N_0} = \bar{X}_i(\ell)$ for all $i \neq 2$, design 2 will still be $\hat{b}$ at the end of the $\ell$th iteration. Continue this process and a false selection is certain when the algorithm terminates. Finally, the probability of $\mathcal{E}$ can be bounded from below as follows.
\begin{align*}
\mathbb{P}(\mathcal{E}) &= \mathbb{P}\left(\bigcap_{i=1}^K E_i \right) = \prod_{i=1}^K \mathbb{P}(E_i) \ge \frac{1}{2} \prod_ {i \neq 2} \mathbb{P}(E_i)\\
&= \frac{1}{2} \prod_{i \neq 2}\Bigg [\underbrace{\mathbb{P}\left\{\bar{X}_{i,N_0} \le \mu_2 - \bar{\eta}-1 \right\}}_{\text{$(\dagger)$}} \underbrace{\mathbb{P} \left\{S_{i, N_0} \le \frac{N_0(\sigma_2 - \bar{\epsilon})}{T} \right\}}_{\text{$(\dagger\dagger)$}} \Bigg],
\end{align*}
where the last equality follows from the independence of $\bar{X}_{i, N_0}$ and $S_{i, N_0}$. Furthermore, $(\dagger) \ge p_i$ for some constant $p_i>0$ (independent of $T$), and $(\dagger\dagger) \ge [\mathcal{K}_i (\sigma_2 - \bar{\epsilon}) / T]^{N_0 -1}$ by Lemma \ref{lemma3}, where $\mathcal{K}_i > 0$ are constants independent of $T$. Gather all the terms and the conclusion follows.
 
\endproof

The key to proving Theorem \ref{thm3.1} is to exploit the asymmetry of the standard deviation estimator's distribution. Specifically, when constructing events $E_i(T)$, we require $S_{i,N_0}$ to decrease in order $1/T$ as $T \rightarrow \infty$ for all $i\neq 2$. Then, Lemma \ref{lemma3} can be used to show a polynomial lower bound for $\mathbb{P}(E_i)$. Another way to construct a ``freezing'' event is by increasing $S_2(\ell)$ in order $T$, but this merely produces an exponential lower bound according to (\ref{eq:lemma1.2}) in Lemma \ref{lemma1}. In other words, only exploiting the left tail of $S_{i,N_0}$ would produce a tighter lower bound for the PFS. 

Theorem \ref{thm3.1} can be counterintuitive at first glance. Recall from \cite{glynn2004large} that for normal samples, any fixed fractions $\alpha_i > 0$ would guarantee an exponential convergence rate. This particularly includes equal allocation, i.e., $N_i = \lfloor T/K \rfloor$ for all designs $i$. In this regard, Theorem \ref{thm3.1} seems to suggest that equal allocation is better than more sophisticated sequential allocation procedures, which contradicts numerous empirical studies in which OCBA exhibits significant advantage over equal allocation. To resolve the ``conflict'', note that the LD rate is only defined in an asymptotic sense, meaning that when $T$ gets large enough, equal allocation will eventually achieve a lower PFS than all three OCBA-type algorithms we consider. However, the crossing point of $T$ may be so large that the PFS is already very close to 0, which also explains why such a crossing point is not always observed in numerical results.

\subsection{A Modification for Improvement} \label{sec3.3}
We propose a simple modification to the three OCBA-type algorithms, which is to make $N_0$ grow linearly in $T$. This can be done by choosing a constant $\alpha_0 \in (0,1/K)$ and setting $N_0 = \lfloor \alpha_0 T \rfloor$. Intuitively, the PFS should converge at least as fast as equally allocating $\lfloor \alpha_0 T \rfloor$ to all designs, where an exponential convergence is guaranteed. More formally, we have the following finite-sample bound on the PFS.
\begin{theorem} \label{thm:improve}
Let Assumption \ref{assump1} hold and suppose that $\mu_1 > \mu_2 > \cdots > \mu_K$. If $N_0 = \lfloor \alpha_0 T\rfloor$ for some $\alpha_0\in (0,1)$, then for OCBA, OCBA-D and OCBA-R, there exists some positive constants $C_1, \ldots, C_K$  (independent of $T$) such that
\begin{equation} \label{eq:PFS-ub}
\text{PFS}(T) \le C_1 \exp\left(-\frac{\delta^2 \alpha_0 T}{8 \sigma_1^2 K} \right) + \sum_{i=2}^K C_i \exp\left( -\frac{\bar{\delta}_i^2 \alpha_0 T}{2\sigma_i^2 K} \right), \quad \forall T \ge KN_0,
\end{equation}
where $\delta := \mu_1 - \mu_2$ and $\bar{\delta}_i = \mu_2 - \mu_i + \frac{\delta}{2}$ for $i = 2, \ldots, K$.
\end{theorem}

\proof{Proof of Theorem \ref{thm:improve}}
Note that since $N_0 \le N_i \le T$ for all designs $i$, if the event
\begin{equation*}
\mathcal{E} := \bigcap_{r=N_0}^{T} \left\{ \left\{\bar{X}_{1,r} \ge \mu_1 - \frac{\delta}{2} \right\} \bigcap \left\{ \bigcap_{i \neq 1} \left\{\bar{X}_{i,r} \le \mu_i + \bar{\delta}_i \right\} \right\}  \right\}
\end{equation*}
occurs, then we have a correct selection regardless of the exact values of $N_i$'s. Apply a Gaussian tail bound for $\bar{X}$ and we have
\begin{align*}
\text{PFS}(T) \le \mathbb{P} (\mathcal{E}^c) &\le \sum_{r = N_0}^{T} \left[\mathbb{P}\left(\bar{X}_{1,r} < \mu_1 - \frac{\delta}{2} \right) + \sum_{i=2}^K \mathbb{P} \left(\bar{X}_{i, r} > \mu_i + \bar{\delta}_i \right) \right]\\
&\le \sum_{r = N_0}^\infty \mathbb{P}\left(\bar{X}_{1,r} < \mu_1 - \frac{\delta}{2} \right) + \sum_{i=2}^K \sum_{r=N_0}^\infty \mathbb{P}\left(\bar{X}_{i, r} > \mu_i + \bar{\delta}_i \right)\\
&\le \sum_{r=N_0}^\infty \exp\left(-\frac{\delta^2 r}{8 \sigma_1^2} \right) + \sum_{i=2}^K \sum_{r=N_0}^\infty \exp\left(-\frac{\bar{\delta}_i^2 r}{2\sigma_i^2} \right).
\end{align*}
Evaluate the geometric sums and (\ref{eq:PFS-ub}) follows.
 
\endproof
The bound (\ref{eq:PFS-ub}) fills the long-standing void of a finite-sample PFS upper bound for OCBA-type algorithms. It also applies to a broad class of algorithms that involve a warm-up phase of acquiring initial estimates. An idea similar to using a linearly increasing $N_0$ is to enforce hard thresholds for the actual fractions such that, e.g., $N_i(\ell) / \sum_{j} N_j(\ell) > \epsilon_i$ for some $\epsilon_i > 0$. Both methods will force $N_i(\ell)$ to grow at least linearly fast in $T$, but we work with the former mainly for conveniently obtaining a PFS bound. The choice of $\alpha_0$ inevitably involves a tradeoff between lower initial estimation error and higher flexibility in subsequent allocation. In Section \ref{sec5}, we will use numerical results to demonstrate that an appropriately chosen $\alpha_0$ can lead to a significant improvement in the finite-sample PCS.

One drawback of the finite-sample bound in (\ref{eq:PFS-ub}) is that it is too general and thus can be quite loose. While a tighter upper bound should reflect the pros and cons of different sequential allocation strategies, deriving such a bound is known to be very challenging even for nicely structured fully sequential algorithms. In the upcoming section, we turn our attention to algorithms which follow simple designs yet capture some key features of advanced algorithms. The idea is to examine the individual impact of a feature  through LD rate analysis, and keep the intuition uncluttered from other common features in a sophisticated algorithm.

\section{Characterizing the LD Rate for Simplified Algorithms} \label{sec4}
In this section, we focus on algorithms with an exponential convergence rate, for which the LD rate is one of the most precise quantitative measures of asymptotic behavior. Nevertheless, LD rate analysis remains difficult for sophisticated sequential allocation algorithms. In this section, we \emph{exactly characterize} the LD rate for some simplified algorithms, and compare their LD rates with that achieved by the optimal static allocation derived in \cite{glynn2004large}. Each algorithm to be considered has a simple structure yet represents certain important feature of more advanced algorithms. Our analysis will focus on a two-design case for better tractability, but the proof techniques and insights can provide a basis for more general convergence analysis.

\subsection{Algorithms Overview} \label{sec4.1}
We consider a case of $K = 2$ and study three algorithms, which are presented in Algorithms \ref{DS}-\ref{two-phase}. Algorithm \ref{DS} is the deterministic algorithm studied in \cite{glynn2004large}, which statically allocates the budget according to pre-specified fractions $p$ and $1-p$, hence the name ``deterministic static (DS)''. 

A slight modification of DS leads to the randomized static (RS) algorithm in Algorithm \ref{RS}, which uses the static fractions as a sampling probability distribution at every iteration, and thus can be roughly regarded as a simplified version of OCBA-R or the Top-two Sampling Algorithms in \cite{russo2016simple}. To the best of our knowledge, the (frequentist) convergence rate of such a randomized algorithm has not been studied in the literature. 

Finally, Algorithm \ref{two-phase} is a two-phase algorithm which uses phase I to estimate the optimal DS fractions (see Section \ref{sec4.2} for more details), and then implements the estimated fractions in phase II. The two-phase algorithm is a vanilla version of our modified OCBA-type algorithms, as it enforces a linearly growing $N_0$, but does not update the fraction estimates in all subsequent iterations. Also, notice that we do not reuse the initial $2N_0$ samples in phase II, so $N_1$ and $N_2$ are not bounded from below by a linear function of $T$ and Theorem \ref{thm:improve} does not apply. However, we will show that it still has an exponential convergence rate due to the rapid decrease of initial estimation error as $T$ increases.

\begin{algorithm}[htb]
\caption{Deterministic static (DS)}\label{DS}
\begin{algorithmic}[1]
\State {\bf Input:} $p \in (0,1), T \ge 0$.
\State {\bf Allocation:} Run $N_1 = \lfloor pT \rfloor$ and $N_2 = \lfloor (1-p)T \rfloor$ replications for designs 1 and 2.
\State {\bf Output:} $\argmax_{i \in \{1,2\}} \bar{X}_{i, N_i}$.
\end{algorithmic}
\end{algorithm}

\begin{algorithm}[htb]
\caption{Randomized static (RS)}\label{RS}
\begin{algorithmic}[1]
\State {\bf Input:} $p \in (0,1), T \ge 0$.
\State {\bf Allocation:} At each iteration, independently simulate design 1 with probability $p$, and design 2 with probability $1-p$. 
\State {\bf Output:} $\argmax_{i \in \{1,2\}} \bar{X}_{i, N_i}$.
\end{algorithmic}
\end{algorithm}

\begin{algorithm}[htb]
\caption{Two-phase}\label{two-phase}
\begin{algorithmic}[1]
\State {\bf Input:} $\alpha_0 \in (0,1), T \ge 0$.
\State {\bf Phase I:} Run $N_0=\lfloor \alpha_0 T/2 \rfloor$ replications for each design and compute $S_{1, N_0}$ and $S_{2,N_0}$. Set $\hat{p} \gets S_{1,N_0} / (S_{1,N_0} + S_{2,N_0})$.
\State {\bf Phase II:} Run $N_1 = \lfloor (1-\alpha_0)\hat{p}T \rfloor$ additional replications from design 1, and $N_2 = \lfloor (1-\alpha_0)(1-\hat{p})T \rfloor$ additional replications from design 2.
\State {\bf Output:} $\argmax_{i \in \{1,2\}} \bar{X}_{i, N_i}$.
\end{algorithmic}
\end{algorithm}

\subsection{LD Rate Analysis} \label{sec4.2}
\subsubsection{Analysis of DS Algorithm} \label{sec4.2.1}
Before proceeding to the LD rate analysis of the RS and two-phase algorithms, we recall some well-established results for the DS algorithm. In addition, we also derive a few new results which will serve as benchmarks. Following the normality assumption and letting $\delta:=\mu_1 - \mu_2$, the PFS can be written as
\begin{equation*}
\text{PFS}_{\text{DS}}(T) =  \mathbb{P}\left(\bar{X}_{1,N_1} - \bar{X}_{2,N_2} < 0\right) =\int_{-\infty}^{-\frac{\delta}{\sqrt{\frac{\sigma_1^2}{N_1} + \frac{\sigma_2^2}{N_2} }}} \frac{1}{\sqrt{2 \pi}} e^{-\frac{t^2}{2}} dt,
\end{equation*}
where we assume that $\mu_1 > \mu_2$ and $\sigma_1, \sigma_2 > 0$ henceforth. Ignoring the integrality constraints on $N_1$ and $N_2$, it can be shown that setting $N_1 / (N_1+N_2) \approx p^* := \sigma_1 / (\sigma_1 + \sigma_2)$ minimizes the PFS. In simulation literature, this is often known as ``the optimal strategy for two-design problems is to allocate the budget proportionally to their standard deviations''. The same conclusion can be reached by maximizing the following LD rate with respect to $p \in (0,1)$, 
\begin{equation*}
-\lim_{T \rightarrow \infty} \frac{1}{T} \log\text{PFS}_{\text{DS}}(T) = \frac{\delta^2}{2\left(\frac{\sigma_1^2}{p} + \frac{\sigma_2^2}{1-p} \right)},
\end{equation*}
where $p^*$ is again the unique maximizer, and the corresponding optimal LD rate is given by
\begin{equation} \label{eq:rate-DS}
R_{\text{DS}}^*(\delta, \sigma_1, \sigma_2) := \frac{\delta^2}{2(\sigma_1+\sigma_2)^2}.
\end{equation}
We will use the optimal DS allocation as a benchmark in subsequent analysis. In practice, the true variances are unknown and thus $p^*$ cannot be implemented. A simple alternative is equal allocation (EA), i.e., setting $p = 1/2$. The LD rate for EA is given by
\begin{equation}  \label{eq:rate-EA}
R_{\text{EA}}(\delta, \sigma_1, \sigma_2) := \frac{\delta^2}{4(\sigma_1^2 + \sigma_2^2)}.
\end{equation}
Since $2(\sigma_1 + \sigma_2)^2 \le 4(\sigma_1^2 + \sigma_2^2) \le 4(\sigma_1+\sigma_2)^2$, we have $R_{\text{DS}}^* /2 \le R_{\text{EA}} \le R_{\text{DS}}^*$. In other words, EA's LD rate is never more than a factor of 2 away from the optimal DS rate. Another interesting fact is that, without prior knowledge on the designs' performance, EA is the most robust algorithm. Indeed, consider the robust optimization problem
\begin{equation*}
\inf_{p\in[0,1]} \sup_{\sigma_1, \sigma_2>0} \left\{ \frac{R^*(\delta, \sigma_1, \sigma_2)}{R_{\text{EA}}(\delta, \sigma_1, \sigma_2)} = \frac{\frac{\sigma_1^2}{p} + \frac{\sigma_2^2}{1-p}}{(\sigma_1 + \sigma_2)^2} \right\},
\end{equation*}
which is to find the $p$ that minimizes the worst case ratio between $R_{\text{DS}}^*$ and $R_{\text{EA}}$. It can be checked that the inner-layer problem's optimal value is $\min\{\frac{1}{p}, \frac{1}{1-p}\}$, so $p=1/2$ is the optimal solution. 

We now derive a PFS bound that holds for an important class of algorithms. Since the optimal DS allocation only involves the designs' variance information, it would be reasonable for us to restrict our discussion to algorithms with the following property.
\begin{definition}
A fully sequential algorithm is called \emph{variance-driven} if 
\begin{enumerate}
\item[(i)]
at iteration $\ell = 0$, it runs $N_0 \ge 2$ replications for each design to obtain initial variance estimates $S^2_1(0)$ and $S^2_2(0)$;
\item[(ii)]
at every iteration $\ell = \bar{\ell}$, the algorithm decides which design to simulate next solely based on $\left(S_1^2(\ell), S_2^2(\ell)\right)$ for all $\ell \le \bar{\ell}$, i.e., the history of variance estimates up to iteration $\bar{\ell}$;
\item[(iii)]
at the end of final iteration $\ell = L$, output $\hat{b} = \argmax_{i \in \{1,2\}} \bar{X}_{i, N_i(L)}$.
\end{enumerate}
\end{definition}

Note that in the case of $K=2$, OCBA's allocation fractions in (\ref{eq3.2}) degenerate to $\alpha_1/\alpha_2 =\sigma_1/\sigma_2$. Therefore, the three OCBA-type algorithms we considered in Section \ref{sec3}, i.e., OCBA, OCBA-D and OCBA-R, all fall into the category of variance-driven on two-design problems. We will derive a tight PFS upper bound which holds for all algorithms of this type. 

\begin{lemma} \label{lemma4}
Let $\bar{X}_n$ and $S_n^2$ be the sample mean and sample variance of $n$ i.i.d. normal random variables, respectively. Then, for all $n \ge 2$, $\bar{X}_n$ is independent of $(S_2^2, S_3^2, \ldots, S_n^2)$.
\end{lemma}

Lemma \ref{lemma4} is an extension of the well-known result that $\bar{X}_n$ and $S^2_n$ are independent for normal distribution. Given a total budget of $T$, we let $N_1$ and $N_2$ denote the total number of simulation runs for designs 1 and 2 when the algorithm terminates, i.e., $N_1+N_2 = T$. Then, Lemma \ref{lemma4} has the following implication in our context.

\begin{corollary} \label{cor1}
For any variance-driven algorithm, it holds that $(\bar{X}_{1, N_1}, \bar{X}_{2,N_2}) \mid N_1 \sim (Z_1, Z_2)$ almost surely, where $Z_1\sim \mathcal{N}(\mu_1, \sigma_1^2/N_1)$, $Z_2\sim\mathcal{N}(\mu_2, \sigma_2^2/(T-N_1))$, and $Z_1$ is independent of $Z_2$.
\end{corollary}

Generally speaking, the final mean estimates $\bar{X}_{1,N_1}$ and $\bar{X}_{2,N_2}$ can have a highly non-trivial correlation if an algorithm sequentially allocates the computing budget based on some iteratively updated statistics. Surprisingly, Corollary \ref{cor1} reveals that for variance-driven algorithms, $\bar{X}_{1,N_1}$ and $\bar{X}_{2,N_2}$ are conditionally independent given $N_1 = n_1$ for some $n_1 \ge N_0$. Moreover, their joint distribution coincides with what we get from deterministically allocating $n_1$ and $T-n_1$ runs to designs 1 and 2, respectively. This is due to the nice property of normal distribution characterized in Lemma \ref{lemma4}, and it gives rise to a tight PFS lower bound for all variance-driven algorithms.

\begin{proposition} \label{prop2}
For any variance-driven algorithm $\mathcal{A}$, we have
\begin{equation*}
\text{PFS}_{\mathcal{A}}(T) \ge 1 - \Phi\left(\frac{\delta \sqrt{T}}{\sigma_1 + \sigma_2} \right), \quad \forall T \ge 2N_0,
\end{equation*}
where $\Phi$ is the cumulative distribution function (c.d.f.) of $\mathcal{N}(0,1)$ distribution.
\end{proposition}
\proof{Proof of Proposition \ref{prop2}.}
For any fixed $N_0\le N_1\le T-N_0$, 
\begin{equation*}
\text{PFS} = \mathbb{P}\left(\bar{X}_{1,N_1} < \bar{X}_{2,T-N_1}\right) = 1 - \Phi\left(\frac{\delta}{\sqrt{\frac{\sigma_1^2}{N_1} + \frac{\sigma_2^2}{T-N_1}}} \right),
\end{equation*}
where the right-hand side (RHS) is convex in $N_1$ and is minimized when $N_1 = N_1^* := p^* T$. For an algorithm $\mathcal{A}$ described in the statement, it follows from Corollary \ref{cor1} that $(\bar{X}_{1,N_1}, \bar{X}_{2,N_2}) \mid N_1$ is distributed as two independent normal random variables. Thus, by Jensen's inequality,
\begin{equation*}
\text{PFS}_{\mathcal{A}}(T) = \E\left[\mathbb{P}\left(\bar{X}_{1, N_1} < \bar{X}_{2,N_2} \mid N_1\right) \right] \ge \mathbb{P} \left(\bar{X}_{1, \E N_1} < \bar{X}_{2, \E N_2}\right),
\end{equation*}
where the RHS is further bounded from below by the PFS corresponding to $N_1^*$, which yields exactly the lower bound in the statement.
\endproof

Proposition \ref{prop2} establishes optimality for the optimal DS algorithm in a very strong sense: no variance-driven algorithm can beat the optimal DS allocation under any finite $T$ (up to some rounding error). The same typically does not hold if $K \ge 3$, where it can be checked numerically that the optimal DS algorithm may perform poorly on some problem instances under small budgets. Nonetheless, from an asymptotic point of view, it remains an open question whether sequential algorithms can achieve a higher LD rate than the optimal DS algorithm when $K \ge 3$.

\subsubsection{Analysis of RS Algorithm} \label{sec4.2.2}
Recall from Algorithm \ref{RS} that at each iteration, the RS algorithm simulates design 1 with probability (w.p.) $p$ and design 2 w.p. $1-p$, where $p \in (0,1)$ and the samples are independent of the decisions. Let $\left\{U(\ell)\right\}$ be a sequence of i.i.d. $\text{Bernoulli}(p)$ random variables representing whether design 1 is sampled at each iteration $\ell$. To ensure that the sample means are well-defined, we set $N_1 = \sum_{\ell=1}^T U(\ell) +1, N_2 = T - N_1 + 1$ so that each design gets sampled at least once. Then, the PFS is given by
\begin{align} 
\text{PFS}_{\text{RS}}(T) &= \mathbb{P}(\bar{X}_{1,N_1} < \bar{X}_{2, N_2}) = \E[\mathbb{P}(\bar{X}_{1,N_1} < \bar{X}_{2, N_2} \mid N_1)] \notag\\
&= \sum_{k=0}^T \mathbb{P}\left(\bar{X}_{1, k+1} < \bar{X}_{2, (T-k)+1}\right) {n \choose k} p^k (1-p)^{T-k}, \label{eq:PFS-RS}
\end{align}
which does not allow a closed form. However, a quick observation is that the RHS of (\ref{eq:PFS-RS}) is bounded from below by the term corresponding to $k=0$, i.e.,
\begin{equation*}
\text{PFS}_{\text{RS}}(T) \ge \mathbb{P}\left(\bar{X}_{1, 1} < \bar{X}_{2, T+1}\right) (1-p)^T, \quad \forall T,
\end{equation*}
which gives the LD rate upper bound
\begin{equation*}
-\lim_{T \rightarrow \infty}\frac{1}{T}\log\text{PFS}_{\text{RS}}(T) \le -\log(1-p).
\end{equation*}
Thus, the RS algorithm's LD rate is bounded as $\delta \rightarrow \infty$, which is in sharp contrast with the LD rate of the DS algorithm, where the latter grows in order $\delta^2$ according to (\ref{eq:rate-DS}). Since the separation margin of $\mu_1$ and $\mu_2$ measures the difficulty of a correct selection, this means that the RS algorithms cannot take advantage of a larger $\delta$ due to the randomness introduced in allocation. It also echoes our observation in Section \ref{sec3} that algorithms with the same limiting allocation fractions may have drastically different LD rates. More precisely, we have the following exact characterization.

\begin{theorem} \label{thm:RS}
For the RS algorithm, we have
\begin{equation} \label{eq:rate-RS}
-\lim_{T \rightarrow \infty}\frac{1}{T}\log\text{PFS}_{\text{RS}}(T)= \inf_{\alpha \in [0,1]} \left\{\frac{\delta^2}{2\left(\frac{\sigma_1^2}{\alpha}+\frac{\sigma_2^2}{1-\alpha}\right)} + kl(\alpha, p)\right\},
\end{equation}
where $kl(\alpha, p):= \alpha \log\frac{\alpha}{p} + (1-\alpha) \log\frac{1-\alpha}{1-p}$ is the Kullback-Leibler (K-L) divergence between two Bernoulli distributions with parameters $\alpha$ and $p$, respectively.
\end{theorem}

The optimization problem in (\ref{eq:rate-RS}) is in general non-convex and an analytical solution is not available. Nonetheless, it can be checked numerically that the $p$ value maximizing the LD rate of the RS algorithm is different from $p^*$, the optimal DS fraction. The proof of Theorem \ref{thm:RS} relies on the following lemma, where $\Z^+$ denotes the set of nonnegative integers.

\begin{lemma} \label{lemma:uniform}
Let $g_T: \Z^+ \mapsto (0, \infty)$ be a sequence of functions for $T \in \Z^+$. If there exists a function $g^*: (0,1) \mapsto \R$ such that $\frac{1}{T} \log g_T(\lfloor \alpha T \rfloor)$ converges uniformly to $g^*(\alpha)$ on $\alpha \in [0,1]$, then
\begin{equation} \label{eq:uniform}
\lim_{T\rightarrow \infty} \frac{1}{T}  \log \left[\sum_{k=0}^T g_T(k)\right] = \sup_{\alpha \in [0,1]} g^*(\alpha).
\end{equation}
\end{lemma}

\proof{Proof of Lemma \ref{lemma:uniform}.}
First of all, notice that
\begin{equation*}
\lim_{T\rightarrow \infty} \frac{1}{T} \log \left[\sum_{k=0}^T g_T(k)\right] \ge \lim_{T\rightarrow \infty} \frac{1}{T}\log g_T(\lfloor \alpha T \rfloor) = g^*(\alpha), \quad \forall \alpha \in [0,1].
\end{equation*}
so taking the supremum on the RHS gives a lower bound. For the upper bound,
\begin{align*}
\lim_{T\rightarrow \infty} \frac{1}{T} \log\left[\sum_{k=0}^T g_T(k)\right] &\le \lim_{T\rightarrow \infty} \sup_{\alpha \in [0,1]} \frac{1}{T}\log \left[(T+1) g_T(\lfloor \alpha T \rfloor) \right]\\
& = \sup_{\alpha \in [0,1]} \lim_{T\rightarrow \infty} \log \frac{1}{T}\left[(T+1) g_T(\lfloor \alpha T \rfloor) \right]\\
& = \sup_{\alpha \in [0,1]} g^*(\alpha),
\end{align*}
where the interchange of limit and supremum is justified by the uniform convergence of $\frac{1}{T} \log g_T(\lfloor \alpha T \rfloor)$ (see, e.g., Theorem 5.3 in \cite{shapiro2014lectures}).  
\endproof

Lemma \ref{lemma:uniform} can be viewed as a generalization of $\lim_{t\rightarrow \infty}\frac{1}{t}\log(e^{-at} + e^{-bt}) = b$ for $a>b>0$, i.e., the LD rate of a sum is determined by the largest summand. In the context of Theorem \ref{thm:RS}, $g_T(\cdot)$ is a sequence of functions that take values in $(0,1)$, so (\ref{eq:uniform}) roughly corresponds to the $g_T(\lfloor \alpha T \rfloor)$ term that converges to 0 ``at the slowest rate''. With Lemma \ref{lemma:uniform}, Theorem \ref{thm:RS} can be shown by checking the uniform convergence of the  function sequence $\frac{1}{T} \log g_T(\lfloor \alpha T \rfloor)$.

\subsubsection{Analysis of Two-phase Algorithm} \label{sec4.2.3}
Recall from Algorithm \ref{two-phase} that the two-phase algorithm first uses $\alpha_0$ fraction of the budget to obtain initial estimates of $\sigma_1$ and $\sigma_2$, and then allocates the remaining budget according to the plug-in estimate of $p^*$ given by $\hat{p} := S_1 / (S_1 + S_2)$. Since the allocation in phase II only depends on the value of $\hat{p}$ from phase I, we first characterize $\hat{p}$'s distribution as follows.
\begin{lemma} \label{lemma:rate-two-phase}
The probability density function (p.d.f.) of $\hat{p}$ is given by
\begin{equation*}
f_{N_0}(p) = \frac{2\Gamma(N_0 - 1)}{\left[\Gamma\left(\frac{N_0 - 1}{2} \right)\right]^2} \left[p(1-p)\right]^{N_0 - 2} \left(\frac{\sigma_1 \sigma_2}{(1-p)^2 \sigma_1^2 + p^2 \sigma_2^2} \right)^{N_0 - 1}, \quad p\in[0,1],
\end{equation*}
where $\Gamma$ is the gamma function $\Gamma(t) = \int_0^\infty x^{t-1}e^{-x} dx$.
\end{lemma}

With Lemma \ref{lemma:rate-two-phase}, we can apply a similar technique involving ``the slowest term'' in Section \ref{sec4.2.3} to establish the LD rate of the two-phase algorithm. 
\begin{theorem} \label{thm:two-phase}
For the two-phase algorithm, we have
\begin{equation} \label{eq:rate-two-phase}
-\lim_{T \rightarrow \infty}\frac{1}{T}\log\text{PFS}(T)= \min_{p \in [0,1]} \left\{\frac{(1-\alpha_0)\delta^2}{2\left(\frac{\sigma_1^2}{p} + \frac{\sigma_2^2}{1-p}\right)} + \frac{\alpha_0}{2} \log \left(\frac{(1-p)^2 \sigma_1^2 + p^2 \sigma_2^2}{2p(1-p) \sigma_1\sigma_2} \right) \right\}.
\end{equation}
\end{theorem}

We argue that the RHS of (\ref{eq:rate-two-phase}) is not bounded in $\delta$. Let $p^*_{\delta}$ be the minimizer corresponding to parameter $\delta$, and let $\underline{p}^*:=\liminf_{\delta \rightarrow\infty} p^*_\delta, \bar{p}^*:=\limsup_{\delta \rightarrow\infty} p^*_\delta$. If $\underline{p}^* = 0$ or $\bar{p}^*=1$, then the second term in the objective function is unbounded in $\delta$; otherwise if $\underline{p}^* > 0$ and $\bar{p}^*< 1$, then the first term will be unbounded as $\delta \rightarrow \infty$. Either case, the two-phase algorithm does not suffer from an LD rate bottleneck as the RS algorithm does. However, it can be checked numerically that the two-phase algorithm is usually far from matching the LD rate of the optimal DS algorithm. This should be no surprise since $\hat{p}$ is subject to estimation error.

\section{Numerical Results} \label{sec5}
We test the performance of four algorithms: OCBA, OCBA+, OCBA-D+, OCBA-R+, where OCBA is the original OCBA with a constant initial sample size $N_0$, and the ``+'' algorithms are modified versions that implement $N_0 = \lfloor \alpha_0 T \rfloor$ for some chosen $\alpha_0 \in (0,1)$. The purpose is to see whether making $N_0$ grow linearly with $T$ can boost the PCS. We apply these algorithms to six problem instances, which are listed as follows. In particular, the ``Slippage Configuration'' refers to the least favorable setting where all the suboptimal designs have the same mean.

\begin{enumerate}
\item[(1)]
\emph{Ten designs A}:
$\mu = [1, 1.1, 1.2, \ldots, 1.8, 5], \sigma = [5, 5, \ldots, 5, 20]$.

\item[(2)]
\emph{Ten designs B}:
$\mu = [1, 1.1, 1.2, \ldots, 1.8, 5], \sigma = [20, 20, \ldots, 20, 5]$.

\item[(3)]
\emph{Slippage Configuration A}:
$\mu = [1, 1, 1, 1, 2], \sigma = [2,2,2,2,10]$.

\item[(4)]
\emph{Slippage Configuration B}:
$\mu = [1, 1, 1, 1, 2], \sigma = [10,10,10,10,2]$.

\item[(5)]
\emph{Equal variances}:
$\mu = [1, 2, \ldots, 10], \sigma_i = 10, \forall i = 1,2 \ldots, 10$.

\item[(6)]
\emph{Increasing variances}:
$\mu = [1,2,\ldots, 10], \sigma = [6, 7, 8, \ldots, 15]$.
\end{enumerate}

The algorithm parameters are $N_0 = 10, \Delta = 20$ for OCBA, and $\alpha_0 = 0.2$ for all the modified algorithms. We would like to see how fast the PCS converges to 1 as $T$ increases from 200 to 4000 (with an increment of 200). To estimate the PCS, all the algorithms are run for 10,000 independent replications using common random numbers, i.e., the algorithms share the same $X_{ir}$ samples for each design. The PCS curves are gathered in Figure \ref{fig1}, and a number of observations follow.

\begin{figure}
    \centering
    \subfloat[Ten designs A.]{{\includegraphics[width=0.45\linewidth]{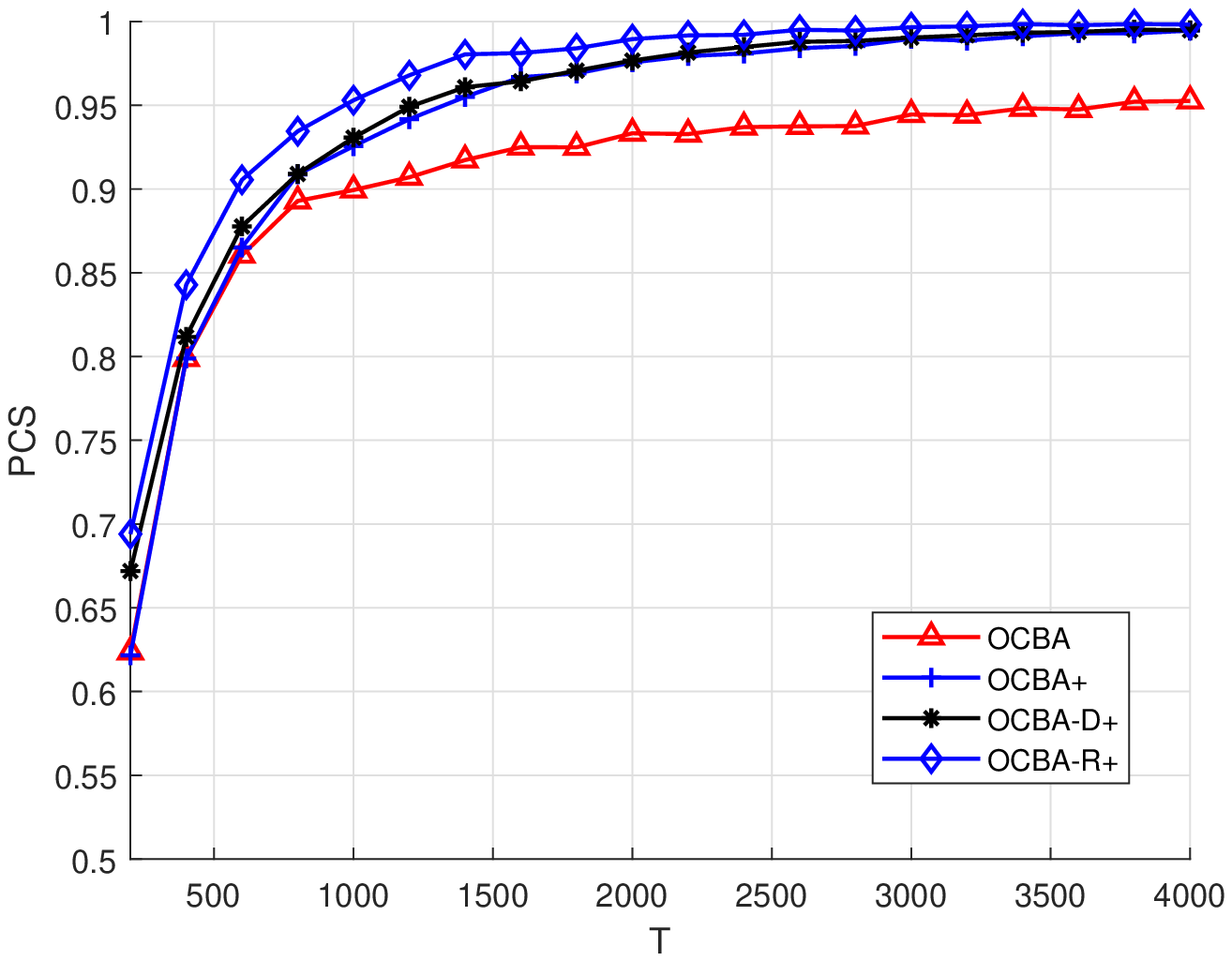} }}
    \quad
    \subfloat[Ten designs B.]{{\includegraphics[width=0.45\linewidth]{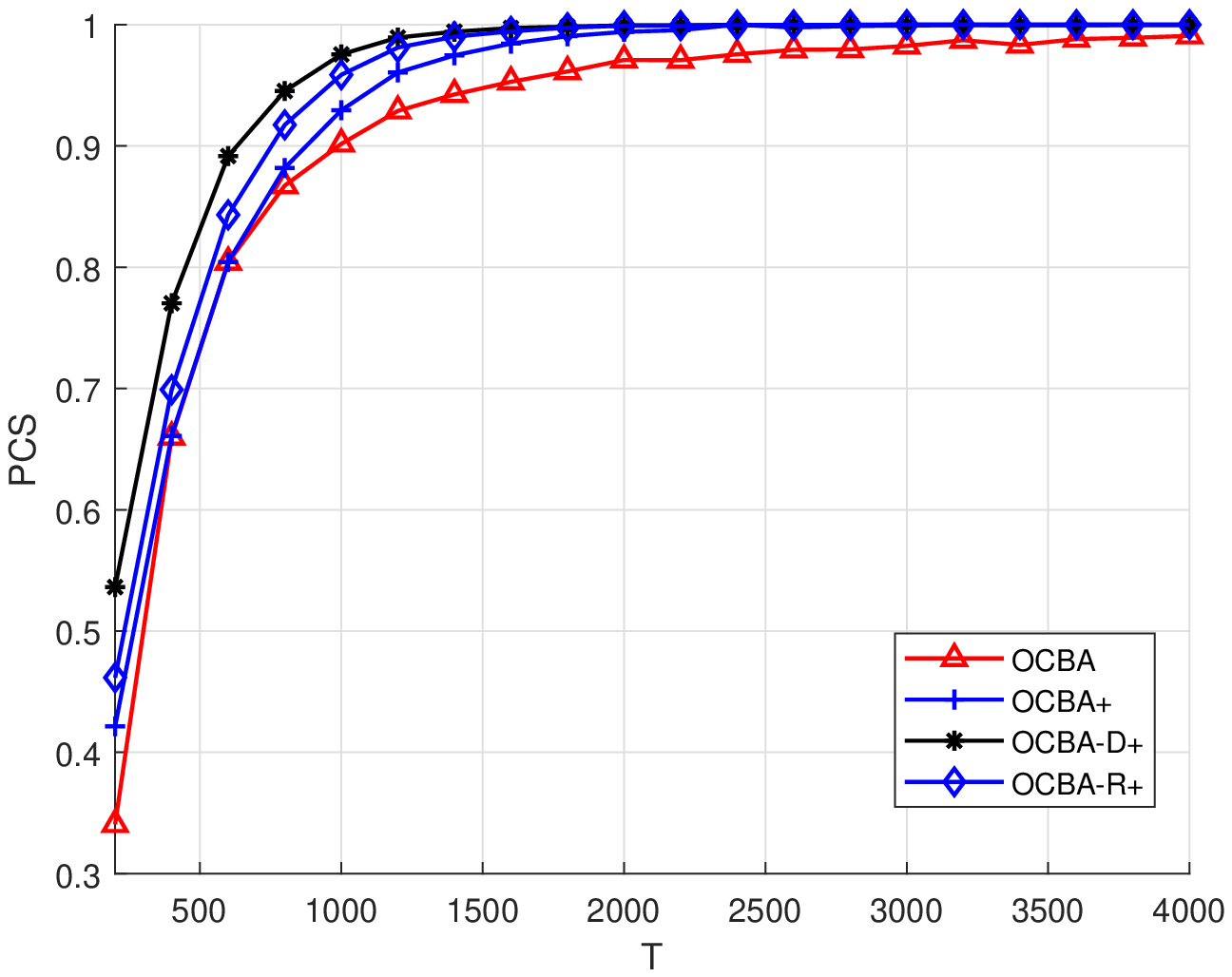}}} \\
    \subfloat[Slippage Configuration A.]{{\includegraphics[width=0.45\linewidth]{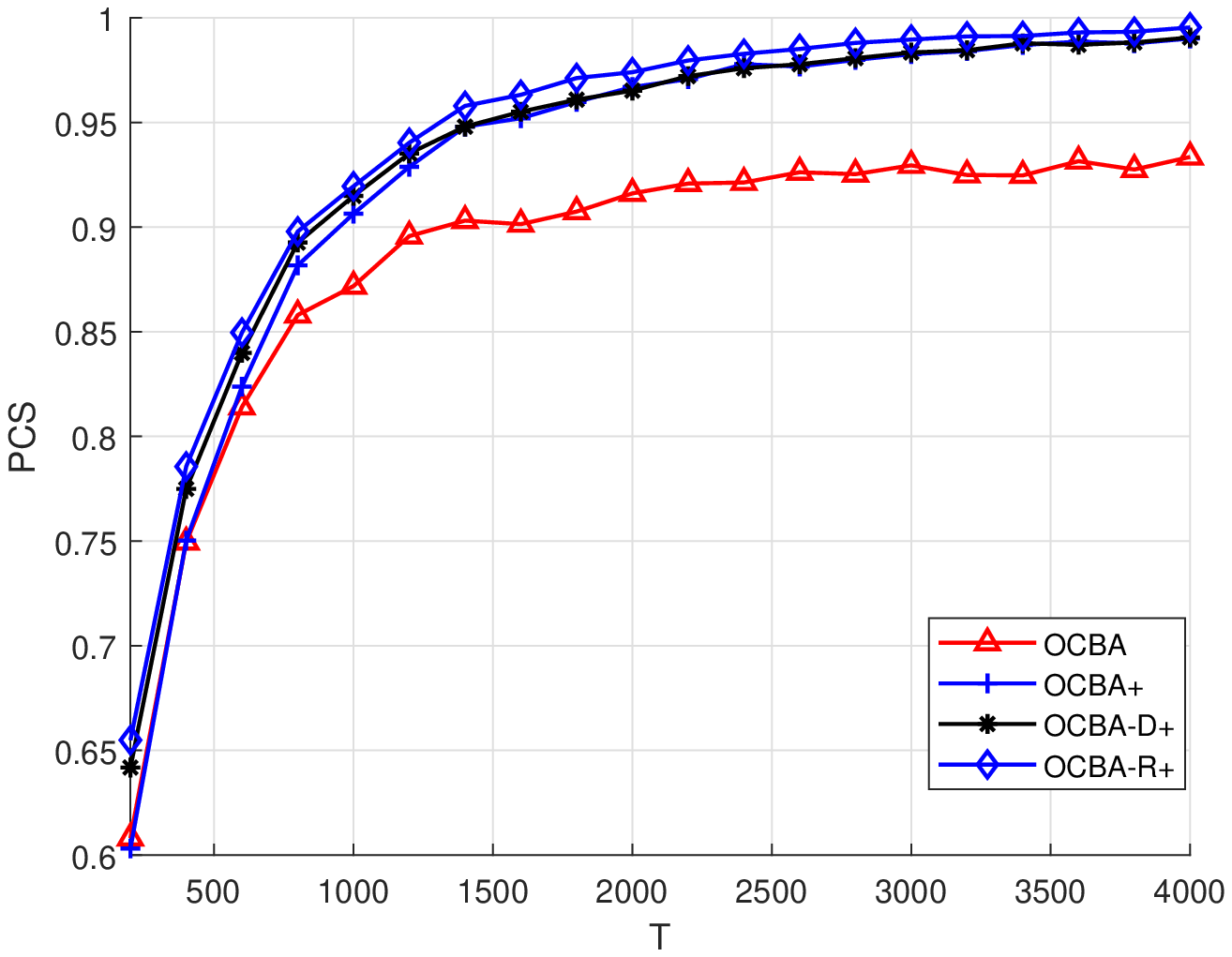} }}
    \quad
    \subfloat[Slippage Configuration B.]{{\includegraphics[width=0.45\linewidth]{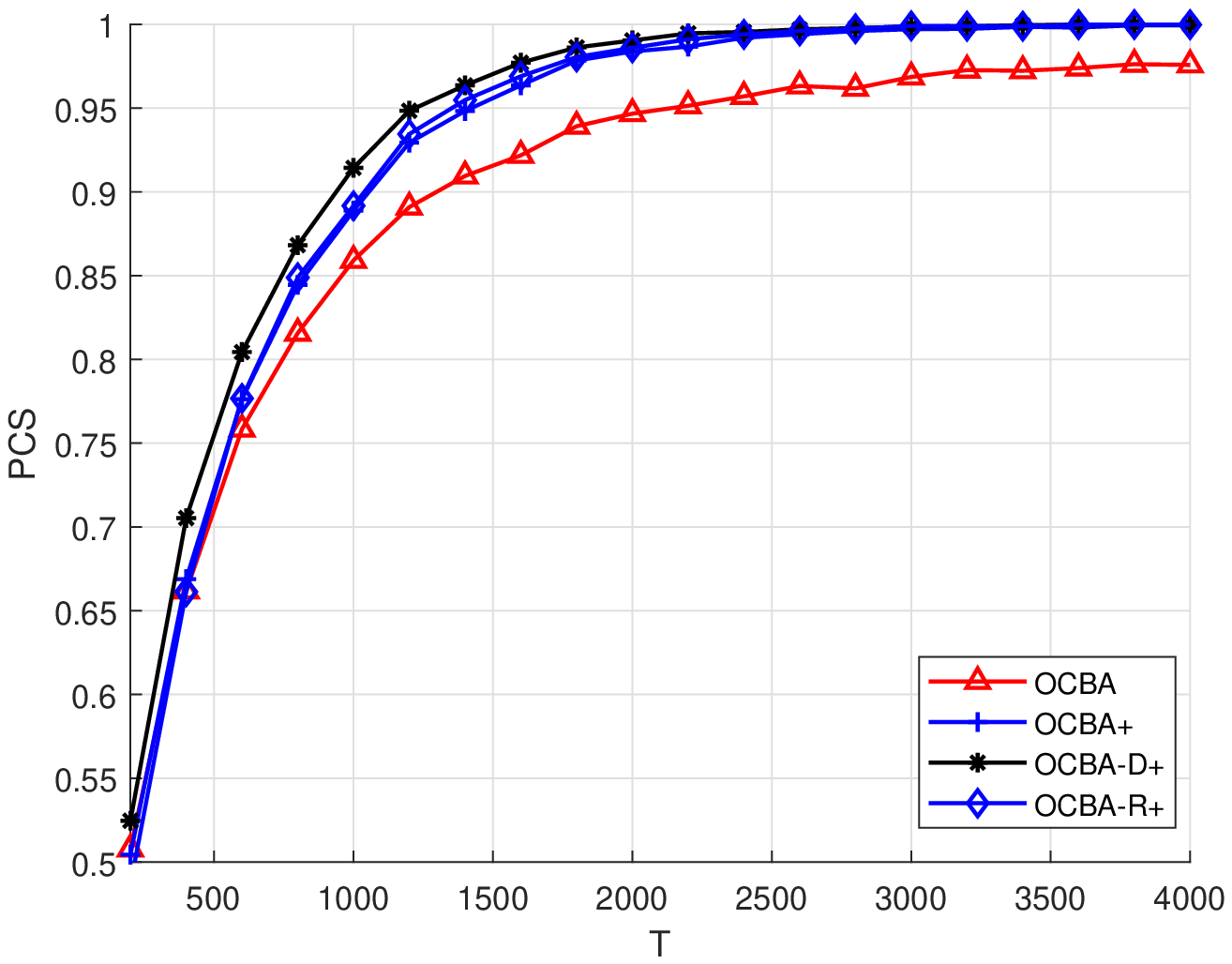} }} \\
     \subfloat[Equal variances.]{{\includegraphics[width=0.45\linewidth]{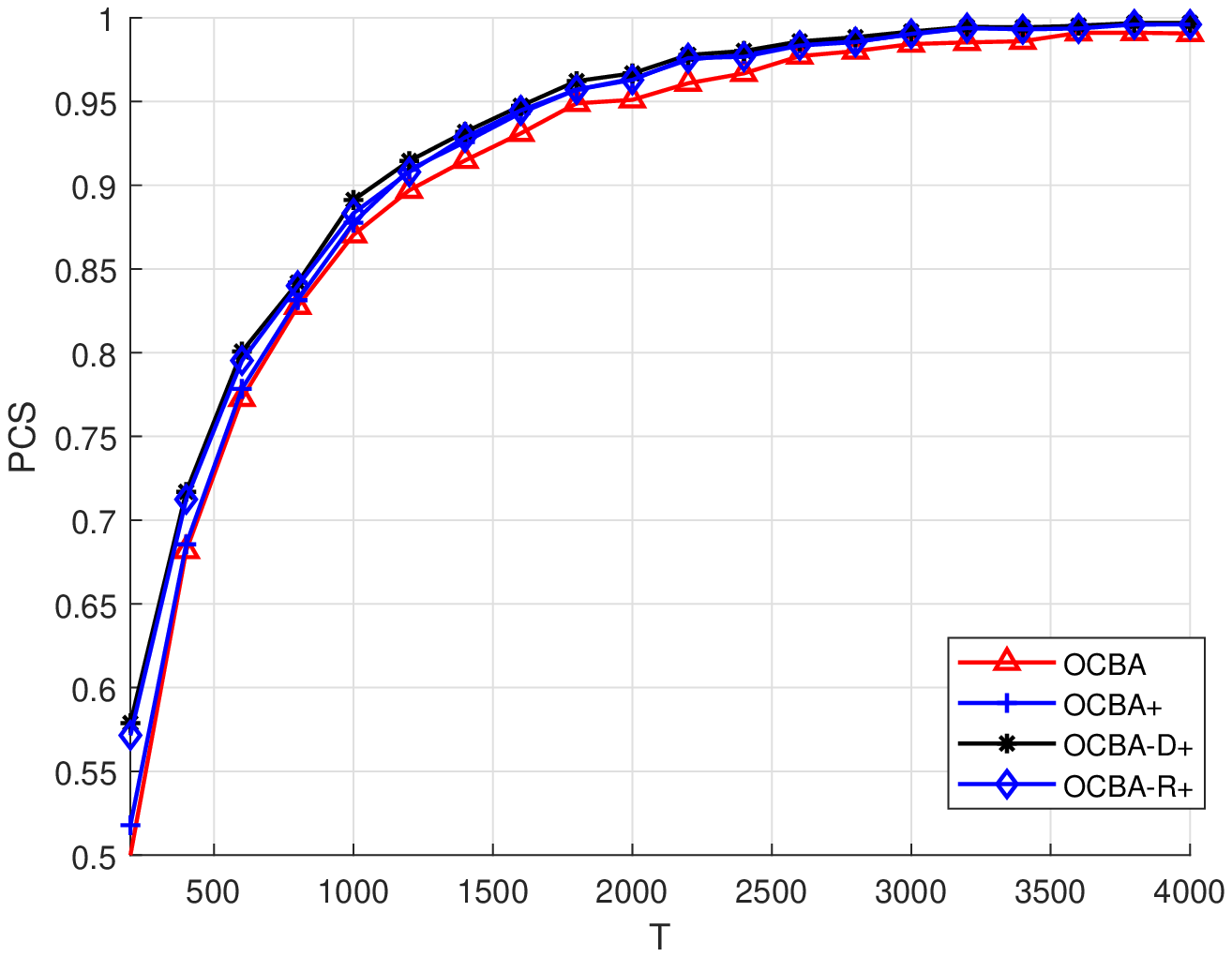} }}
    \quad
    \subfloat[Increasing variances.]{{\includegraphics[width=0.45\linewidth]{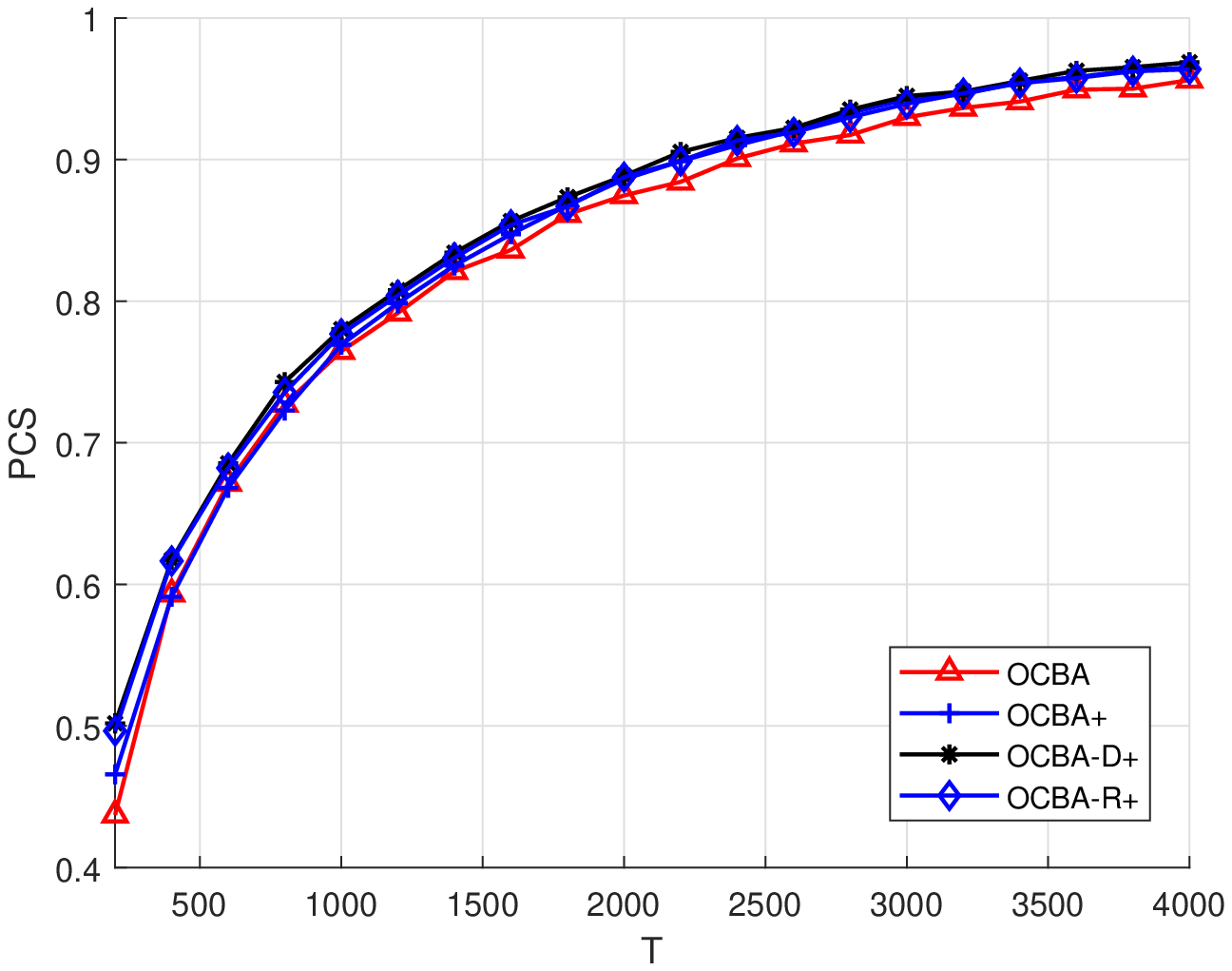} }} \\
    \caption{Comparison of PCS for different algorithms.}%
    \label{fig1}%
\end{figure}

\begin{enumerate}
\item
In Figure \ref{fig1} (a)-(f), the modified algorithms achieve a higher PCS than OCBA for every fixed $T$. This demonstrates the advantage of using a linearly growing $N_0$.
\item
In Figure \ref{fig1} (a) and (c), OCBA clearly suffers from slower convergence. Notably, on those two problem instances, OCBA takes approximately three times the budget of OCBA-R+ to attain a $95\%$ PCS, which echoes our finding that a constant $N_0$ only gives a polynomial rate. 
\item
The improvement is less visible in Figure \ref{fig1} (e) and (f), where the corresponding instances are ``easier'' as they have relatively spread-out means and smaller variances. Therefore, more benefit can be gained from using linearly growing $N_0$ on ``harder'' problem instances.
\item
Among the modified algorithms, OCBA-D+ and OCBA-R+ outperform OCBA+ on all problem instances, which can be expected from their fully sequential feature as it makes better use of the budget. However, no clear ranking is observed between OCBA-D+ and OCBA-R+.
\end{enumerate}

\section{Conclusion and Future Work} \label{sec6}
This paper studies the problem of fixed budget Ranking and Selection with independent normal samples.
By analyzing the performance of several OCBA-type algorithms, we discover that a constant initial sample size only leads to a sub-exponential convergence rate. A linearly growing initial sample size is then proposed to achieve both better theoretical (asymptotic) and practical (finite-sample) performance. In addition, we explicitly characterize the large deviations rate of some simplified algorithms, which sets a basis for more general convergence analysis.

With the study of fixed budget Ranking and Selection actively ongoing, we think that this work points to a least a number of directions that are worth pursuing.
\begin{enumerate}
\item
{\bf Tighter PFS bounds.} Our paper explores some techniques for analyzing the convergence rate of sequential allocation algorithms. However, more powerful approaches still need to be developed to derive tighter PFS bounds which capture the features of different algorithms.
\item
{\bf Explicit balance between exploration and exploitation.} The fixed budget Ranking and Selection problem is also called a ``pure exploration'' problem in the Multi-Armed Bandits literature, as the tradeoff between exploration and exploitation is relatively less explicit. It would be practically useful to design an algorithm that can explicitly balance these two aspects.
\item
{\bf Better performance measure.} As is mentioned at the end of Section \ref{sec2}, the three prevailing performance measures, i.e., finite-sample bounds, large deviations rate and numerical results, are all subject to some restrictions. This calls for a better performance measure, which should not only reflect the general performance of algorithms, but also allow tractable characterization.
\end{enumerate}

\section*{Acknowledgement} The authors gratefully acknowledge the support by the National Science Foundation under Grant CAREER CMMI-1453934.

\appendix
\section{Proofs for OCBA-type Algorithms.}
\proof{Proof of Proposition \ref{prop3.1}.}
Since (ii) is a direct consequence of (i) due to the consistency of mean and variance estimators, we will show (i) and (iii) for OCBA, OCBA-D, and OCBA-R.
\begin{enumerate}
\item {\bf Proof for OCBA.}
\begin{enumerate}
\item[(i)]
Suppose that all the random variables are defined on a probability space $(\Omega, \mathcal{F}, \mathbb{P})$. By the Strong Law of Large Numbers (SLLN), there exists a measurable set $\Omega_1 \subseteq \Omega$ such that $\mathbb{P}(\Omega_1) = 1$ and for all $\omega \in \Omega_1$ and any design $i \in \mathcal{I}$,
\begin{enumerate}
\item[(1)]
$\bar{X}_{i,n} \rightarrow \mu_i$ and $S_{i,n} \rightarrow \sigma_i$ as $n \rightarrow \infty$.
\item[(2)]
$S_{i,n} > 0$ for all $n \ge 2$.
\item[(3)]
$\bar{X}_{i, n} \neq \bar{X}_{j, m}$ for all $j \neq i$ and $n, m \ge 1$.
\end{enumerate}
Since the samples follow nondegenerate normal distributions, (2) and (3) both occur with probability 0, and $\Omega_1$ is guaranteed to exist. Here we mainly need (2) and (3) to avoid some trivial edge cases for OCBA. 

Take any sample path $\omega \in \Omega_1$. We will show that on $\omega$, $N_i(\ell) \rightarrow \infty$ as $\ell \rightarrow \infty$ for all designs. Assume for a contradiction that this does not hold, then there exists a nonempty set $\tilde{\mathcal{I}} \subseteq \mathcal{I}$ such that $N_j(\ell) \not\rightarrow \infty$ for all $j \in \tilde{\I}$. This means that for all $j \in \tilde{\I}$, $\bar{X}_j(\ell)$ and $S_j(\ell)$ will be fixed at some constants for all $\ell$ large enough. Under this assumption, we claim that the following holds. 
\begin{claim} \label{cm3.1}
There exists a constant $\tilde{\alpha} > 0$ such that for all designs $i \in \I$, $\hat{\alpha}_i(\ell) > \tilde{\alpha}$ for all $\ell$ large enough.
\end{claim}
\begin{proof}{Proof of Claim \ref{cm3.1}}
It suffices to show that every $\hat{\alpha}_i(\ell)$ converges to some positive constant. Note that $\bar{X}_i(\ell) \rightarrow \mu_i$ and $S_i(\ell) \rightarrow \sigma_i$ as $\ell \rightarrow \infty$ for all $i \in \I \setminus \tilde{\I}$. Since $\mu_i \neq \mu_j$ for all $i \neq j$, we further know that for all $\ell$ large enough, $\hat{b}$ will be fixed and so does the form of $\hat{\beta}_i$ (note that $\hat{\beta}_{\hat{b}}$ has a different form than the other $\hat{\beta}_i$'s). It then follows from the continuity of $\hat{\beta}_i$ in $(\bar{X}_i, S_i)_{i \in \I}$ that $\hat{\beta}_i(\ell)$ will converge to some constant $\tilde{\beta}_i > 0$, and $\hat{\alpha}_i(\ell) \rightarrow \tilde{\beta}_i / (\sum_{i \in \I} \tilde{\beta}_i) > 0$.
\end{proof}
However, when $T'(\ell)$ gets sufficiently large, we would have $\lfloor \hat{\alpha}_j T'(\ell) \rfloor \ge \lfloor \tilde{\alpha} T'(\ell) \rfloor > N_j(\ell)$ for some design $j \in \tilde{\I}$, where step 6 of Algorithm \ref{OCBA} will allocate additional budget to $j$, hence a contradiction.
\item[(iii)]
Similarly, we will show convergence on $\Omega_1$. Let $\epsilon$ be an arbitrary positive number. From (ii), there exists $\ell_0$ such that $\hat{\alpha}_i(\ell) \in [\alpha_i - \epsilon, \alpha_i + \epsilon]$ for all $\ell  \ge \ell_0$. Furthermore, since $N_i(\ell) \rightarrow \infty$ for any design $i$, we can find $\ell_i \ge \ell_0$ such that $N_i(\ell_i+1) > N_i(\ell_i)$, i.e., $N_i$ jumps at the $\ell_i$th iteration. Let $\bar{\ell} := \max_{i \in \I} \ell_i$ so that every $N_i$ has jumped at least once since iteration $\bar{\ell}$. We claim that for all $\ell \ge \bar{\ell}$,
\begin{equation} \label{eq1}
\lfloor (\alpha_i - \epsilon) T'(\ell) \rfloor \le N_i(\ell) \le \lfloor (\alpha_i + \epsilon) T'(\ell) \rfloor, \quad \forall i \in \I.
\end{equation}
To see why (\ref{eq1}) holds, first notice from step 7 of Algorithm \ref{OCBA} that 
\begin{equation*}
N_i(\ell) = \max\left\{N_i(\ell-1), \lfloor \hat{\alpha}_i(\ell) T'(\ell) \rfloor \right\} \ge \lfloor \hat{\alpha}_i(\ell) T'(\ell) \rfloor \ge \lfloor (\alpha_i - \epsilon) T'(\ell) \rfloor,
\end{equation*}
so the first inequality in (\ref{eq1}) holds. For the other inequality, there are two cases to consider. If $N_i(\ell) = \lfloor \hat{\alpha}_i(\ell) T'(\ell) \rfloor$ then (\ref{eq1}) holds apparently. Otherwise, $N_i(\ell)  = \lfloor \hat{\alpha}_i(\ell') T'(\ell') \rfloor$, where 
\begin{equation*}
\ell' := \max\{\tilde{\ell} < \ell \mid  N_i(\tilde{\ell}+1) > N_i(\tilde{\ell})\},
\end{equation*}
which is the iteration corresponding to the most recent jump. We know $\ell_i$ corresponds to a jump, so $\ell_i \le \ell' < \ell$. Since $\ell_i \ge \ell_0$, the definition of $\ell_0$ ensures that $\hat{\alpha}_i(\ell') \le \alpha_i + \epsilon$. Thus,
\begin{equation*}
N_i(\ell) = \lfloor \hat{\alpha}_i(\ell') T'(\ell') \rfloor \le \lfloor (\alpha_i + \epsilon) T'(\ell) \rfloor,
\end{equation*}
so (\ref{eq1}) always holds. This would imply that 
\begin{equation*}
\frac{\alpha_i - \epsilon}{1 + K \epsilon} \le \liminf_{\ell \rightarrow \infty} \frac{N_i(\ell)}{\sum_{j \in \I} N_j(\ell)} \le \limsup_{\ell \rightarrow \infty} \frac{N_i(\ell)}{\sum_{j \in \I} N_j(\ell)} \le \frac{\alpha_i + \epsilon}{1 - K \epsilon}
\end{equation*}
for all $\epsilon$ sufficiently small. Send $\epsilon \rightarrow 0$ and the conclusion follows.
\end{enumerate}

\item{\bf Proof for OCBA-D.}
\quad
\begin{enumerate}
\item[(i)]
Using the $\Omega_1$ constructed previously, this can be shown by following a similar argument as in OCBA's proof (proof by contradiction). 
\item[(iii)]
Pick design 1 as reference design. It suffices to show that on $\Omega_1$, $N_i(\ell) / N_1(\ell) \rightarrow \alpha_i / \alpha_1$ as $\ell \rightarrow \infty$ for all $i \neq 1$. To begin with, for any $\epsilon >0$, there exists $\ell'$ such that for all $\ell \ge \ell'$, $\hat{\alpha}_i(\ell) \in [\alpha_i - \epsilon, \alpha_i + \epsilon], \forall i \in \I$. Furthermore, since $N_i(\ell)$ is nondecreasing in $\ell$ and $N_i(\ell) \rightarrow \infty$ as $\ell \rightarrow \infty$, we can find $\ell'' \ge \ell'$ such that all the designs satisfy $N_i(\ell) \ge N_i(\ell') + 2$ for all $\ell \ge \ell''$, i.e., all the $N_i$'s have jumped at least twice since the $\ell'$th iteration. Then, we claim that for any design $i \neq 1$,
\begin{equation} \label{eq2}
\frac{\alpha_i + \epsilon}{N_i(\ell) - 1} \ge \frac{\alpha_1 - \epsilon}{N_1(\ell)}, \quad \forall \ell \ge \ell''.
\end{equation}
Assume for a contradiction that (\ref{eq2}) does not hold. Let 
$$\ell_i := \max\{\tilde{\ell} \mid N_i(\tilde{\ell}) = N_i(\ell) - 1 \},$$
namely, the iteration when design $i$ is chosen to be simulated and $N_i$ is about to jump from $N_i(\ell)-1$ to $N_i(\ell)$. It then follows from the definition of $\ell''$ that $\ell' \le \ell_i \le \ell$, and from step 5 of Algorithm \ref{OCBA-D} we have
\begin{equation} \label{eq3}
\frac{\hat{\alpha}_i(\ell_i)}{N_i(\ell_i)} \ge \frac{\hat{\alpha}_j(\ell_i)}{N_j(\ell_i)}, \quad j \neq i.
\end{equation}
However, if (\ref{eq2}) does not hold, then we will have
\begin{equation*}
\frac{\hat{\alpha}_i(\ell_i)}{N_i(\ell_i)} \le \frac{\alpha_i + \epsilon}{N_i(\ell) - 1} < \frac{\alpha_1 - \epsilon}{N_1(\ell)} \le \frac{\hat{\alpha}_1(\ell_i)}{N_1(\ell_i)},
\end{equation*}
hence a contradiction to (\ref{eq3}). Thus, (\ref{eq2}) provides an upper bound on $N_i(\ell)$.
\begin{equation} \label{eq4}
N_i(\ell) \le \left(\frac{\alpha_i + \epsilon}{\alpha_1 - \epsilon}\right) N_i(\ell) + 1.
\end{equation}
By symmetry (using design $i$ as a reference design), we also have
\begin{equation*}
\frac{\alpha_1 + \epsilon}{N_1(\ell) - 1} \ge \frac{\alpha_i - \epsilon}{N_i(\ell)}, \quad \forall \ell \ge \ell'',
\end{equation*}
which gives the following lower bound on $N_i(\ell)$.
\begin{equation} \label{eq5}
N_i(\ell) \ge \left(\frac{\alpha_i - \epsilon}{\alpha_1 + \epsilon}\right) (N_1(\ell) - 1).
\end{equation}
Combining (\ref{eq4}) and (\ref{eq5}) and we have
\begin{equation*}
\frac{\alpha_i - \epsilon}{\alpha_1 + \epsilon} \le \liminf_{\ell \rightarrow \infty} \frac{N_i(\ell)}{N_1(\ell)} \le \limsup_{\ell \rightarrow \infty} \frac{N_i(\ell)}{N_1(\ell)} \le \frac{\alpha_i + \epsilon}{\alpha_1 - \epsilon}.
\end{equation*}
Take $\epsilon \rightarrow 0$ and we conclude that $N_i(\ell) / N_1(\ell) \rightarrow \alpha_i / \alpha_1$ as $\ell \rightarrow \infty$.
\end{enumerate}

\item{\bf Proof for OCBA-R.}
\quad
\begin{enumerate}
\item[(i)]
Let $\mathbbm{1}_{\{\cdot\}}$ denote an indicator function. Define an event
\begin{equation*}
\Omega_2 := \bigcap_{n \ge 1} \bigcap_{0 \le m \le n-1} \left\{\omega \in \Omega \bigg\vert \lim_{\ell \rightarrow \infty} \frac{1}{\ell}\sum_{\ell} \mathbbm{1}_{\{U(\ell) \in [\frac{m}{n}, \frac{m+1}{n}]\}} = \frac{1}{n} \right\},
\end{equation*}
which has probability 1 due to SLLN. We will show that $N_i(\ell) \rightarrow \infty$ as $\ell \rightarrow \infty$ on $\Omega_3 := \Omega_1 \cap \Omega_2$. Assume for a contradiction that this is not true. By an argument similar to OCBA's proof, $\hat{\alpha}_i(\ell) \rightarrow \tilde{\alpha}_i$ as $\ell \rightarrow \infty$ for some $\tilde{\alpha}_i > 0$. Then, any design $i$ will be simulated if and only if $U(\ell)$ falls into some nonempty interval $I_i(\ell)$. We can find $n_i,m_i,k_i \ge 0$ such that $[\frac{m_i}{n_i}, \frac{m_i+k_i}{n_i}] \subseteq I_i(\ell)$ for all $\ell$ large enough, and the definition of $\Omega_2$ ensures that
\begin{equation*}
\liminf_{\ell \rightarrow \infty} \frac{1}{\ell} \sum_{\ell} \mathbbm{1}_{\{U(\ell) \in I_i\}} \ge 
\lim_{\ell \rightarrow \infty} \frac{1}{\ell} \sum_{\ell} \mathbbm{1}_{\left\{U(\ell) \in \left[\frac{m_i}{n_i}, \frac{m_i+k_i}{n_i}\right] \right\}} = \frac{k_i}{n_i},
\end{equation*}
so all the designs will be simulated infinitely often, which is a contradiction. 

\item[(iii)]
It follows from (i) that for any arbitrary $\epsilon>0$, there exists $\ell'$ such that for all designs, $\hat{\alpha}_i(\ell) \in [\alpha_i-\epsilon, \alpha_i+\epsilon]$ for all $\ell \ge \ell'$. Meanwhile, at each iteration $\ell$, design $i$ is simulated if and only if $U(\ell) \in I_i(\ell):= [\sum_{j=0}^{i-1} \hat{\alpha}_j(\ell), \sum_{j=0}^{i} \hat{\alpha}_j(\ell)]$, where we let $\hat{\alpha}_0(\ell) := 0, \forall \ell$. Also define $I^{-\epsilon}_i := [\sum_{j=0}^{i-1} \alpha_j + (i-1)\epsilon, \sum_{j=0}^{i} \alpha_j - i \epsilon]$ and $I^{\epsilon}_i := [\sum_{j=0}^{i-1} \alpha_j - (i-1)\epsilon, \sum_{j=0}^{i} \alpha_j + i \epsilon]$. We may assume that $\epsilon$ is sufficiently small so that $I^{-\epsilon}_i$ and $I^\epsilon_i$ are both well-defined. It follows that $I^{-\epsilon}_i  \subseteq I_i(\ell) \subseteq I^{\epsilon}_i $ for all $\ell \ge \ell'$. Furthermore, there exists intervals $I'_i$ and $I''_i$ (independent of $\ell$) with rational endpoints such that $I'_i \subseteq I^{-\epsilon}_i \subseteq I^\epsilon_i \subseteq I''_i$, $| I^{-\epsilon}_i  \setminus  I'_i| \le \epsilon$ and $| I^{''}_i \setminus I^{\epsilon}_i| \le \epsilon$. Combining all these and by the definition of $\Omega_2$,
\begin{align*}
\alpha_i - 2i \epsilon &= \lim_{\ell\rightarrow \infty}\frac{1}{\ell} \sum_{\ell} \mathbbm{1}_{\{U(\ell) \in I'_i \}} \le \liminf_{\ell\rightarrow \infty}\frac{1}{\ell}\sum_{\ell} \mathbbm{1}_{\{U(\ell) \in I_i(\ell)\}} \\
&\le \limsup_{\ell\rightarrow \infty}\frac{1}{\ell}\sum_{\ell} \mathbbm{1}_{\{U(\ell) \in I_i(\ell)\}} \le \lim_{\ell\rightarrow \infty}\frac{1}{\ell} \sum_{\ell} \mathbbm{1}_{\{U(\ell) \in I''_i \}}  = \alpha_i + 2i \epsilon.
\end{align*}
Send $\epsilon \rightarrow 0$ and we have $\frac{1}{\ell}\sum_{\ell} \mathbbm{1}_{\{U(\ell) \in I_i(\ell)\}} \rightarrow \alpha_i$ as $\ell \rightarrow \infty$. The conclusion follows immediately from the fact that $N_i(\ell) = N_0 + \sum_{\ell \ge 0} \mathbbm{1}_{\{U(\ell) \in I_i(\ell)\}} $.  
\end{enumerate}

\end{enumerate}
\endproof

\proof{Proof of Lemma \ref{lemma1}.}
According to Lemma 1 in \cite{laurent2000adaptive}, if $X\sim \chi^2(n)$, then
\begin{align*}
\mathbb{P}(X - n \le -2\sqrt{nx}) &\le e^{-x}, \quad \forall x > 0,\\
\mathbb{P}(\chi^2(n) - n \ge 2 \sqrt{nx} + 2x) &\le e^{-x}, \quad \forall x >0.
\end{align*}

Since $(n-1)S_n / \sigma^2 \sim \chi^2(n-1)$, (\ref{eq:lemma1.1}) and (\ref{eq:lemma1.2}) can be derived by a change of variable.  
\endproof

\proof{Proof of Lemma \ref{lemma2}}
Fix $\epsilon$ as some arbitrary number in $(0, \sigma)$. From Lemma \ref{lemma1} we know that $\mathbb{P} \left\{S_n \le \sigma-\epsilon \right\} \le C_\epsilon e^{-\gamma_\epsilon n}, \forall \epsilon \in (0, \sigma)$, where $\gamma_\epsilon := \frac{1}{4} \left[1 - \left(\frac{\sigma-\epsilon}{\sigma} \right)^2 \right]^2$ and $C_\epsilon := e^{\gamma_\epsilon}$. Thus, $\exists L \ge 1$ such that
\begin{equation*}
\sum_{n \ge L} \mathbb{P}\left\{S_n \le \sigma - \epsilon \right\} \le \frac{C_\epsilon e^{-\gamma_\epsilon L}}{1 - e^{-\gamma_\epsilon}} \le \frac{c}{2}.
\end{equation*}
If $L=2$, then set $\bar{\epsilon} = \epsilon$ and (\ref{eq:lemma2.1}) holds. Otherwise, since $\mathbb{P} \left\{S_n \le \sigma - \epsilon \right\} \downarrow 0$ as $\epsilon \uparrow \sigma$, $\exists \epsilon' \in (0, \sigma)$ such that $\sum_{n=m}^{L-1} \mathbb{P} \left\{S_n \le \sigma - \epsilon' \right\} \le c/2$. Take $\bar{\epsilon} := \max\{\epsilon, \epsilon' \}$ and we have
\begin{align*}
\sum_{n \ge m} \mathbb{P} \{S_n \le \sigma - \bar{\epsilon} \} &\le \sum_{n=m}^{L-1} \mathbb{P}\{S_n \le \sigma - \bar{\epsilon}\} + \sum_{n \ge L} \mathbb{P} \{S_n \le \sigma - \bar{\epsilon}\} \\
& \le \frac{c}{2} + \sum_{n \ge L} \mathbb{P}\{S_n \le \sigma- \epsilon\} \le c,
\end{align*}
so that (\ref{eq:lemma2.1}) also holds.
 
\endproof

\proof{Proof of Lemma \ref{lemma3}}
Note that $(n-1)S_n^2 / \sigma^2 \sim \chi^2(n-1)$. Let $Z_1, Z_2, \ldots$ be i.i.d. $\mathcal{N}(0,1)$ random variables, and we have
\begin{align*}
\mathbb{P}\{S_n \le a\} &= \mathbb{P}\left\{\frac{(n-1)S_n^2}{\sigma^2} \le \frac{(n-1)a^2}{\sigma^2} \right\}= \mathbb{P} \left\{\sum_{i=1}^{n-1} |Z_i|^2 \le \frac{(n-1)a^2}{\sigma^2} \right\}\\
& \ge \mathbb{P} \left\{\bigcap_{i=1}^{n-1} \left\{|Z_i| \le \frac{a}{\sigma} \right\} \right\} = \left[\mathbb{P}\left\{|Z_1| \le \frac{a}{\sigma}\right\} \right]^{n-1},
\end{align*}
where we observe that
\begin{equation*}
\mathbb{P}\left\{|Z_1| \le \frac{a}{\sigma}\right\} \ge \frac{2a}{\sigma}\frac{1}{\sqrt{2\pi}}e^{-\frac{a^2}{2\sigma^2}}\ge \left(\frac{2e^{-\frac{b^2}{2\sigma^2}}}{\sigma \sqrt{2\pi}}\right) a := \mathcal{K}_b a,
\end{equation*}
by inspecting the shape of normal distribution's density, 
 
\endproof

\proof{Proof of Theorem \ref{thm3.1} (continued).}
It remains to show (\ref{eq:lb1}) for OCBA-D and (\ref{eq:lb2}) for OCBA-R. For OCBA-D, we use the same construction of events $E_i$ as in OCBA's proof. It suffices to show inductively that $\hat{b} = i^*=2$ for all iterations $\ell$. We know that this is true at $\ell = 0$. Assume that it holds for the $(\ell-1)$th iteration. Then, at the $\ell$th iteration, we have for all $i \neq 2$,
\begin{align}
\frac{\hat{\alpha}_{i}(\ell)}{\hat{\alpha}_2(\ell)} = &\frac{S_i^2(\ell) / \hat{\delta}_{2,i}^2(\ell)}{S_2(\ell) \sqrt{\sum_{j \neq 2} \frac{S_j^2(\ell)}{\hat{\delta}_{2,j}^4(\ell)}}}
\le \frac{S_i^2(\ell) / \hat{\delta}_{2,i}^2(\ell)}{S_2(\ell) \sqrt{\frac{S_i^2(\ell)}{\hat{\delta}_{2,i}^4(\ell)}}} \notag\\
=&\frac{S_i(\ell)}{S_2(\ell)} \le \frac{N_0 (\sigma_2 - \bar{\epsilon})}{S_2(\ell) T} \le \frac{N_0}{T}, \label{eq:lb1.1}
\end{align}
which implies that
\begin{equation*}
\frac{\hat{\alpha}_i(\ell)}{N_i(\ell)} = \frac{\hat{\alpha}_i(\ell)}{N_0} \le \frac{\hat{\alpha}_2(\ell)}{T} < \frac{\hat{\alpha}_2(\ell)}{N_2(\ell)}, \quad \forall i \neq 2.
\end{equation*}
Thus, $i^*=2$ and $\hat{b}=2$ at the $\ell$th iteration, and the process will eventually lead to a false selection. The rest follows from the same argument in OCBA's proof.

Next, we prove (\ref{eq:lb2}) for OCBA-R. Fix an arbitrary $\epsilon \in (0,1)$. The event $E_2$ uses the same construction as in OCBA's proof, i.e., 
\begin{equation*}
E_2:= \left\{\bar{X}_{2,n} \ge \mu_2 - \bar{\eta}, \forall n \ge N_0 \right\} \cap \left\{S_{2,n} \ge \sigma_2 - \bar{\epsilon}, \forall n \ge N_0 \right\},
\end{equation*}
where $\bar{\eta}>0$ and $\bar{\epsilon} \in (0, \sigma_2)$ are chosen such that $\mathbb{P}(E_2) \ge 1/2$. For $i \neq 2$, we let
\begin{equation*}
E_i := \left\{\bar{X}_{i, N_0} \le \mu_2 - \eta - 1 \right\} \cap \left\{S_{i, N_0} \le \epsilon(\sigma_2 - \bar{\epsilon})/(K-1) \right\}.
\end{equation*}
Let $\mathcal{E}:= \bigcap_{i=1}^K E_i$ and $L:=T-KN_0-1$. Note that if $\mathcal{E}$ occurs and the algorithm picks $i^*=2$ at all iterations, then a false selection always occurs. This provides a lower bound for the PFS,
\begin{align}
\text{PFS}(T) &\ge \mathbb{P}\left\{\left\{i^*=2, \forall \ell = 0, \ldots, L \right\} \cap \mathcal{E}\right\} \notag\\
&=\mathbb{P}\left\{\left\{i^*=\hat{b}=2, \forall \ell = 0, \ldots, L \right\} \cap \mathcal{E}\right\} \notag\\
&= \E\left[\prod_{\ell=0}^{L} \hat{\alpha}_2(\ell) \mathbbm{1}_{\mathcal{E}} \right], \label{eq:lb2.1}
\end{align}
where the last inequality follow from the design of Algorithm \ref{OCBA-R}, and
\begin{equation*}
\hat{\alpha}_2(\ell) = \frac{S_2(\ell) \sqrt{\sum_{j \neq 2} \frac{S_j^2(\ell)}{\hat{\delta}_{2,j}^4(\ell)}}}{\sum_{j \neq 2} \frac{S_j^2(\ell)}{\hat{\delta}_{2,j}^2(\ell)} + S_2(\ell) \sqrt{\sum_{j \neq 2} \frac{S_j^2(\ell)}{\hat{\delta}_{2,j}^4(\ell)}}}, \quad \forall \ell = 0, \ldots, L,
\end{equation*}
since $\hat{b}=2$ for all $\ell = 0, \ldots, L$. Furthermore, on event $\mathcal{E}$ we have
\begin{equation*}
\hat{\alpha}_i(\ell) \le \frac{\hat{\alpha}_i(\ell)}{\hat{\alpha}_2(\ell)} \le \frac{S_i(\ell)}{S_2(\ell)} = \frac{S_{i,N_0}}{S_2(\ell)} \le \frac{\epsilon(\sigma_2 - \bar{\epsilon})}{(K-1)(\sigma_2 - \bar{\epsilon})} \le \frac{\epsilon}{K-1},
\end{equation*}
where the second inequality follows from (\ref{eq:lb1.1}), the equality follows from $i^*=2$ for all $\ell = 0, \ldots, L$, and the third inequality is a consequence of $E_2$ and $E_i$. Thus, $\hat{\alpha}_2(\ell) = 1 - \sum_{i \neq 2} \hat{\alpha}_i(\ell) \ge 1 - \epsilon$ for $\ell = 0, \ldots, L$, and plugging it into (\ref{eq:lb2.1}) gives
\begin{equation*}
\text{PFS}(T) \ge \mathbb{P}(\mathcal{E}) (1-\epsilon)^{T-KN_0},
\end{equation*}
where $ \mathbb{P}(\mathcal{E}) = \prod_{i=1}^K \mathbb{P}(E_i) > 0$ is a constant independent of $T$. Therefore, 
$$-\lim_{T\rightarrow \infty} \frac{1}{T}\log \text{PFS}(T) \le - \log(1-\epsilon).$$ 
Take $\epsilon \downarrow 0$ and (\ref{eq:lb2}) follows.
 
\endproof

\section{Proofs of Characterizing the LD Rate.}
\proof{Proof of Lemma \ref{lemma4}.}
For any $2 \le k \le n$, $S_k^2$ is a function of the deviations $(\bar{X}_k - X_1, \ldots, \bar{X}_k - X_k)$. Thus, it suffices to show that
\begin{equation*}
\bar{X}_n \perp ((\bar{X}_2 - X_1, \bar{X}_2 - X_2), \ldots, (\bar{X}_n - X_1, \ldots, \bar{X}_n - X_n)),
\end{equation*}
where $\perp$ denotes independence, and we denote the RHS by $Y_n$. Note that $(\bar{X}_n, Y_n)$ is a linear transformation of $(X_1, \ldots, X_n)$ and hence are jointly normal, the result follows from
\begin{equation*}
\text{Cov}(\bar{X}_n, \bar{X}_k - X_j) = 0, \quad \forall 2 \le k \le n, j \le k,
\end{equation*}
which can be checked by direct computation.  
\endproof

\proof{Proof of Corollary \ref{cor1}.}
For $N_0 \le k \le T-N_0$, let 
\begin{equation*}
Y_{1,k}:=(S_{1,N_0}, \ldots, S_{1,k}), \quad Y_{2,k}:=(S_{2,N_0}, \ldots, S_{2,k}).
\end{equation*}
Note that $N_1 = k$ if and only if $(Y_{1,k}, Y_{2, T-k})$ falls into some event $A_k(T)$ that is measurable. Furthermore, following the proof of Lemma \ref{lemma4}, it can be shown that $(\bar{X}_{1,k}, \bar{X}_{2, T-k}) \perp (Y_{1,k}, Y_{2, T-k})$ since $(X_{1,1}, X_{1,2}, \ldots)\perp (X_{2,1}, X_{2,2}, \ldots)$. Thus, for any $N_0 \le k \le T-N_0$,
\begin{align*}
&\mathbb{P}\left(\bar{X}_{1, N_1} \le x,  \bar{X}_{2, N_2} \le y \mid N_1 = k\right)\\
=&\mathbb{P}\left(\bar{X}_{1, k} \le x,  \bar{X}_{2, T-k} \le y \mid (Y_{1,k}, Y_{2, T-k} \in A_k(T)) \right)\\
=&\mathbb{P}\left(\bar{X}_{1, k} \le x,  \bar{X}_{2, T-k} \le y\right).
\end{align*}
The conclusion follows from $\bar{X}_{1,k}\perp \bar{X}_{2,T-k}$.  
\endproof

\proof{Proof of Theorem \ref{thm:RS}.}
Recall the following Gaussian tail bound. For $X \sim \mathcal{N}(0,1)$ and $x > 0$,
\begin{equation} \label{eq:gaussian}
e^{-\frac{x^2}{2}} \ge \mathbb{P}(X >x) \ge \frac{x}{x^2+1}\frac{1}{\sqrt{2 \pi}}e^{-\frac{x^2}{2}}.
\end{equation}
Applying the upper bound in (\ref{eq:gaussian}) to the PFS expression in (\ref{eq:PFS-RS}) yields
\begin{equation} \label{eq:PFS-RS-ub}
\text{PFS}_{\text{RS}}(T) \le \sum_{k=0}^T \exp\left(-\frac{\delta^2}{2 \left(\frac{\sigma_1^2}{k+1} + \frac{\sigma_2^2}{T-k+1} \right)} \right) {n \choose k} p^k (1-p)^{T-k}.
\end{equation}
We will apply Lemma \ref{lemma:uniform} to (\ref{eq:PFS-RS-ub}) to show a lower bound for the RS algorithm's LD rate. The upper bound follows similarly from the lower bound in (\ref{eq:gaussian}) and is thus omitted. Letting $g_T(k)$ denote the summands in the RHS of (\ref{eq:PFS-RS-ub}), we have
\begin{equation} \label{eq:g_T}
\begin{aligned}
 \frac{1}{T} \log g_T(\lfloor \alpha T \rfloor) &= \left\{-\frac{\delta^2}{2\left(\frac{\sigma_1^2 T}{\lfloor \alpha T\rfloor + 1} + \frac{\sigma_2^2 T}{\lfloor (1-\alpha)T \rfloor} +1 \right)} + \frac{1}{T} \log\left(\frac{\Gamma(T+1)}{\Gamma(\lfloor \alpha T \rfloor+1)\Gamma(\lfloor(1-\alpha)T \rfloor+1)} \right)\right.\\
&\quad +\left.\frac{1}{T} \log \left(p^{\lfloor \alpha T \rfloor} \right) + \frac{1}{T} \log\left[(1-p)^{\lfloor (1-\alpha) T\rfloor} \right] \right\},
\end{aligned}
\end{equation}
where the first, third and fourth terms are clearly uniformly convergent in $\alpha$ as $T \rightarrow \infty$. For the second term, a well-known approximation of the log gamma function (see, e.g., \cite{gradshteyn2014table,kowalenko2014exactification}) gives
\begin{equation} \label{eq:gamma}
\log \Gamma(x) = \left(x -\frac{1}{2} \right) \log x - x + \frac{1}{2} \log(2\pi) + r(x),
\end{equation}
where there exists a constant $C>0$ such that the remainder $|r(x)| < C/x$ for all $x \ge 1$. It follows from (\ref{eq:gamma}) that the second term converges uniformly to $-\alpha \log\alpha - (1-\alpha) \log(1-\alpha)$. Sending $T \rightarrow \infty$ in (\ref{eq:g_T}) gives the objective function in the RHS of (\ref{eq:rate-RS}), and applying Lemma \ref{lemma:uniform} leads to the infimum problem.
 
\endproof

\proof{Proof of Lemma \ref{lemma:rate-two-phase}.}
Given two independent nonnegative random variables $X, Y$ with probability densities functions (p.d.f.s) $f_X$ and $f_Y$, it can be verified that for $t\in (0,1)$,
\begin{equation}\label{eq1.3}
\frac{d}{dt}\mathbb{P}\left(\frac{X}{X+Y} \le t\right) = \frac{1}{(1-t)^2} \int_0^\infty y f_X\left(\frac{t}{1-t} y \right)f_Y(y)dy.
\end{equation}
So we only need to compute the p.d.f.s for $S_1$ and $S_2$. Suppose that $N_0 = n$. Then, since $(n - 1) S_1^2 / \sigma_1^2\sim \chi^2(n - 1)$,
\begin{equation*}
\mathbb{P}(S_1 \le x) = \mathbb{P} \left(\chi^2(n - 1) \le \frac{(n - 1) x^2}{\sigma_1^2} \right),
\end{equation*}
and by differentiating on both sides,
\begin{equation*}
f_{S_1}(x) =\frac{2(n - 1) x}{\sigma_1^2 2^{\frac{n-1}{2}} \Gamma\left(\frac{n -1}{2} \right)} \left(\frac{(n-1)x^2}{\sigma_1^2} \right)^{\frac{n-3}{2}} \exp\left(-\frac{(n - 1)x^2}{2\sigma_1^2}  \right).
\end{equation*}
The p.d.f. of $S_2$ has a similar form and is omitted. Plugging their densities into (\ref{eq1.3}) and a direct computation yields the result.
\endproof

\proof{Proof of Theorem \ref{thm:two-phase}.}
The PFS for the two-phase algorithm is given by
\begin{equation} \label{eq:PFS-two-phase}
\text{PFS}(T) = \int_0^1 \mathbb{P}\left( \mathcal{N}\left( \delta, \frac{\sigma_1^2}{\lfloor (1-\alpha_0) pT \rfloor + 1} + \frac{\sigma_2^2}{\lfloor (1-\alpha_0)(1-p)T \rfloor + 1}\right)<0 \right) f_{N_0}(p) dp.
\end{equation}
Apply the Gaussian tail bounds in (\ref{eq:gaussian}) and we have
\begin{equation}\label{eq:PFS-ub-lb}
\int_0^1 \psi_T(p) f_{N_0}(p) dp \ge \text{PFS}(T) \ge K(T) \int_0^1 \psi_T(p) f_{N_0}(p) dp,
\end{equation}
where the function $\psi_T$ is defined as
\begin{equation*}
 \psi_T(p) := \exp\left\{-\frac{\delta^2}{2 \left(\frac{\sigma_1^2}{\lfloor (1-\alpha_0)pT \rfloor + 1} + \frac{\sigma_2^2}{\lfloor (1-\alpha_0)(1-p)T \rfloor + 1} \right)} \right\},
\end{equation*}
and $K(T)$ is of order $\sqrt{T}$ as $T \rightarrow \infty$. From (\ref{eq:PFS-ub-lb}) we know that the PFS has the same LD rate as $\int_0^1 \psi_T(p) f_{N_0}(p) dp$, so it suffices to study the latter. Similar to the proof of Theorem \ref{thm:RS}, if $g_T(p) :=\frac{1}{T} \log\left( \psi_T(p) f_{N_0}(p) \right)$ converges uniformly to a function $g^*$ on $p \in [0,1]$ as $T \rightarrow \infty$, then Theorem 5.3 in \cite{shapiro2014lectures} guarantees that
\begin{equation} \label{eq:interchange}
\lim_{T \rightarrow \infty} \max_{p \in [0,1]}g_T(p) = \max_{p \in [0,1]} g^*(p).
\end{equation}
To check uniform convergence, we have
\begin{align*}
g_T(p) &= \underbrace{-\frac{\delta^2}{2T \left(\frac{\sigma_1^2}{\lfloor (1-\alpha_0)pT \rfloor + 1} + \frac{\sigma_2^2}{\lfloor (1-\alpha_0)(1-p)T \rfloor + 1} \right)}}_{\text{$(1)$}} + \underbrace{\frac{N_0 - 2}{T} \log(p(1-p))}_{\text{$(2)$}}\\
&\underbrace{+\frac{N_0 - 1}{T} \log\left(\sigma_1\sigma_2 / [(1-p)^2\sigma_1^2 +p^2\sigma_2^2] \right)}_{\text{$(3)$}} + \underbrace{\frac{1}{T} \log\left(\frac{2\Gamma(N_0-1)}{\left[\Gamma\left(\frac{N_0-1}{2} \right) \right]^2} \right)}_{\text{$(4)$}},
\end{align*}
where $N_0 = \lfloor \alpha_0 T/2 \rfloor$ and $-\lim_{T\rightarrow \infty} g_T(p)$ gives the objective function in the RHS of (\ref{eq:rate-two-phase}). The uniform convergence of (1), (3) and (4) can be easily checked. However, (2) is not uniformly convergent near the two endpoints 0 and 1. To show (\ref{eq:interchange}), note that (1), (3) and (4) are all uniformly bounded in $p$ and $T$, while $\log(p(1-p)) \rightarrow -\infty$ as $p\rightarrow 0$ or $1$. Thus, for any $M>0$ and $\epsilon>0$, there exists corresponding $0<p_1<p_2<1$ such that for all $T$ large enough,
\begin{equation}\label{eq:g_T-uniform}
g_T(p)  < M-\epsilon, \quad \forall p \in [0, p_1) \cup (p_2, 1].
\end{equation}
Fix a $\tilde{p} \in (0,1)$ and take $M= g^*(\tilde{p})$. Then, $g_T(\tilde{p}) \ge M-\epsilon$ for all $T$ large enough, which together with (\ref{eq:g_T-uniform}) implies that the maximizer of $g_T$ can only lie in $[p_1, p_2]$ for all $T$ large. Let $p^* \in \argmax_{p\in[0,1]}g^*(p)$, which is guaranteed to exist since $g^*$ is continuous on $(0,1)$ and $g^*(p) \rightarrow -\infty$ as $p \rightarrow 0$ or $1$ (see, e.g., Proposition A.8 in \cite{bertsekas1999nonlinear} for Weierstrass' Extreme Value Theorem). Further expand the interval $[p_1, p_2]$ if necessary such that $p^* \in[p_1, p_2]$, which does not affect the preceding argument. Since $g_T$ converges uniformly on $[p_1, p_2]$,  we have
\begin{align*}
\lim_{T\rightarrow \infty}\frac{1}{T} \log \int_0^1 \left( \psi_T(p) f_{N_0}(p) \right)dp & \le \lim_{T \rightarrow \infty} \frac{1}{T} \log \left( \max_{p \in [p_1,p_2]} \psi_T(p) f_{N_0}(p) \right)\\
&\le \lim_{T \rightarrow\infty} \max_{p\in[p_1, p_2]} g_T(p) \\
&= \max_{p\in[p_1, p_2]} g^*(p) = \max_{p \in [0,1]} g^*(p),
\end{align*}
where the last equality follows from $p^* \in [p_1, p_2]$. For a lower bound, choose an $\epsilon > 0$ such that $[p^* - \epsilon, p^* + \epsilon] \subseteq [0,1]$,  and we have
\begin{align*}
\lim_{T\rightarrow \infty}\frac{1}{T} \log \int_0^1 \left( \psi_T(p) f_{N_0}(p) \right)dp &\ge \lim_{T\rightarrow \infty}\frac{1}{T} \log \int_{p^* - \epsilon}^{p^* + \epsilon} \left( \psi_T(p) f_{N_0}(p) \right)dp\\
& \ge \lim_{T \rightarrow \infty}\min_{p \in [p^* - \epsilon, p^* + \epsilon]} \frac{1}{T} \log \left(2 \epsilon \psi_T(p) f_{N_0}(p) \right) \\
&= \min_{p \in [p^* - \epsilon, p^* + \epsilon]} g^*(p),
\end{align*}
where the last equality is due to uniform convergence. Take $\epsilon \downarrow 0$ and the lower bound follows from the continuity of $g^*$ on $(0,1)$.  
\endproof

\bibliographystyle{ieeetr}
\bibliography{improveOCBA}

\begin{thebibliography}{10}

\bibitem{bechhofer1954single}
R.~E. Bechhofer, ``A single-sample multiple decision procedure for ranking
  means of normal populations with known variances,'' {\em The Annals of
  Mathematical Statistics}, pp.~16--39, 1954.

\bibitem{kim2001fully}
S.-H. Kim and B.~L. Nelson, ``A fully sequential procedure for
  indifference-zone selection in simulation,'' {\em ACM Transactions on
  Modeling and Computer Simulation (TOMACS)}, vol.~11, no.~3, pp.~251--273,
  2001.

\bibitem{jeff2006fully}
L.~Jeff~Hong, ``Fully sequential indifference-zone selection procedures with
  variance-dependent sampling,'' {\em Naval Research Logistics (NRL)}, vol.~53,
  no.~5, pp.~464--476, 2006.

\bibitem{frazier2014fully}
P.~I. Frazier, ``A fully sequential elimination procedure for indifference-zone
  ranking and selection with tight bounds on probability of correct
  selection,'' {\em Operations Research}, vol.~62, no.~4, pp.~926--942, 2014.

\bibitem{branke2005new}
J.~Branke, S.~E. Chick, and C.~Schmidt, ``New developments in ranking and
  selection: an empirical comparison of the three main approaches,'' in {\em
  Proceedings of the 37th conference on Winter simulation}, pp.~708--717,
  Winter Simulation Conference, 2005.

\bibitem{kim2007recent}
S.-H. Kim and B.~L. Nelson, ``Recent advances in ranking and selection,'' in
  {\em Simulation Conference, 2007 Winter}, pp.~162--172, IEEE, 2007.

\bibitem{chick2012sequential}
S.~E. Chick and P.~Frazier, ``Sequential sampling with economics of selection
  procedures,'' {\em Management Science}, vol.~58, no.~3, pp.~550--569, 2012.

\bibitem{ma2017efficient}
S.~Ma and S.~G. Henderson, ``An efficient fully sequential selection procedure
  guaranteeing probably approximately correct selection,'' in {\em Simulation
  Conference (WSC), 2017 Winter}, pp.~2225--2236, IEEE, 2017.

\bibitem{chen2000simulation}
C.-H. Chen, J.~Lin, E.~Y{\"u}cesan, and S.~E. Chick, ``Simulation budget
  allocation for further enhancing the efficiency of ordinal optimization,''
  {\em Discrete Event Dynamic Systems}, vol.~10, no.~3, pp.~251--270, 2000.

\bibitem{glynn2004large}
P.~Glynn and S.~Juneja, ``A large deviations perspective on ordinal
  optimization,'' in {\em Proceedings of the 36th conference on Winter
  simulation}, pp.~577--585, Winter Simulation Conference, 2004.

\bibitem{audibert2010best}
J.-Y. Audibert and S.~Bubeck, ``Best arm identification in multi-armed
  bandits,'' in {\em COLT-23th Conference on Learning Theory-2010}, pp.~13--p,
  2010.

\bibitem{carpentier2016tight}
A.~Carpentier and A.~Locatelli, ``Tight (lower) bounds for the fixed budget
  best arm identification bandit problem,'' in {\em Conference on Learning
  Theory}, pp.~590--604, 2016.

\bibitem{russo2016simple}
D.~Russo, ``Simple bayesian algorithms for best arm identification,'' in {\em
  Conference on Learning Theory}, pp.~1417--1418, 2016.

\bibitem{qin2017improving}
C.~Qin, D.~Klabjan, and D.~Russo, ``Improving the expected improvement
  algorithm,'' in {\em Advances in Neural Information Processing Systems},
  pp.~5387--5397, 2017.

\bibitem{lee2004optimal}
L.~H. Lee, E.~P. Chew, S.~Teng, and D.~Goldsman, ``Optimal computing budget
  allocation for multi-objective simulation models,'' in {\em Proceedings of
  the 36th conference on Winter simulation}, pp.~586--594, Winter Simulation
  Conference, 2004.

\bibitem{jia2013efficient}
Q.-S. Jia, E.~Zhou, and C.-H. Chen, ``Efficient computing budget allocation for
  finding simplest good designs,'' {\em IIE Transactions}, vol.~45, no.~7,
  pp.~736--750, 2013.

\bibitem{gao2017robust}
S.~Gao, H.~Xiao, E.~Zhou, and W.~Chen, ``Robust ranking and selection with
  optimal computing budget allocation,'' {\em Automatica}, vol.~81, pp.~30--36,
  2017.

\bibitem{gao2017new}
S.~Gao, W.~Chen, and L.~Shi, ``A new budget allocation framework for the
  expected opportunity cost,'' {\em Operations Research}, vol.~65, no.~3,
  pp.~787--803, 2017.

\bibitem{peng2018ranking}
Y.~Peng, E.~K. Chong, C.-H. Chen, and M.~C. Fu, ``Ranking and selection as
  stochastic control,'' {\em IEEE Transactions on Automatic Control}, vol.~63,
  no.~8, pp.~2359--2373, 2018.

\bibitem{ryzhov2016convergence}
I.~O. Ryzhov, ``On the convergence rates of expected improvement methods,''
  {\em Operations Research}, vol.~64, no.~6, pp.~1515--1528, 2016.

\bibitem{lee2010review}
L.~H. Lee, C.~Chen, E.~P. Chew, J.~Li, N.~A. Pujowidianto, and S.~Zhang, ``A
  review of optimal computing budget allocation algorithms for simulation
  optimization problem,'' {\em International Journal of Operations Research},
  vol.~7, no.~2, pp.~19--31, 2010.

\bibitem{wu2018ocba}
D.~Wu and E.~Zhou, ``Provably improving the optimal computing budget allocation
  algorithm,'' in {\em Simulation Conference (WSC), 2018 Winter} (M.~Rabe,
  A.~A. Juan, N.~Mustafee, A.~Skoogh, S.~Jain, and B.~Johansson, eds.),
  (Piscataway, New Jersey), Institute of Electrical and Electronics Engineers,
  Inc., 2018.

\bibitem{chen2011stochastic}
C.-h. Chen and L.~H. Lee, {\em Stochastic simulation optimization: an optimal
  computing budget allocation}, vol.~1.
\newblock World scientific, 2011.

\bibitem{law1991simulation}
A.~M. Law, W.~D. Kelton, and W.~D. Kelton, {\em Simulation modeling and
  analysis}, vol.~2.
\newblock McGraw-Hill New York, 1991.

\bibitem{bechhofer1995design}
R.~G. Bechhofer, {\em Design and analysis of experiment for statistical
  selection, screening, and multiple comparisons}.
\newblock No.~04; QA279, B4., 1995.

\bibitem{shapiro2014lectures}
A.~Shapiro, D.~Dentcheva, {\em et~al.}, {\em Lectures on stochastic
  programming: modeling and theory}, vol.~16.
\newblock SIAM, 2014.

\bibitem{laurent2000adaptive}
B.~Laurent and P.~Massart, ``Adaptive estimation of a quadratic functional by
  model selection,'' {\em Annals of Statistics}, pp.~1302--1338, 2000.

\bibitem{gradshteyn2014table}
I.~S. Gradshteyn and I.~M. Ryzhik, {\em Table of integrals, series, and
  products}.
\newblock Academic press, 2014.

\bibitem{kowalenko2014exactification}
V.~Kowalenko, ``Exactification of stirling's approximation for the logarithm of
  the gamma function,'' {\em arXiv preprint arXiv:1404.2705}, 2014.

\bibitem{bertsekas1999nonlinear}
D.~P. Bertsekas, {\em Nonlinear programming}.
\newblock Athena scientific Belmont, 1999.

\end{thebibliography}

\end{document}